\documentclass[11pt,psamsfonts,reqno]{amsart}
\usepackage{amsmath}
\usepackage{amsthm}
\usepackage{amssymb}
\usepackage{amscd}
\usepackage{amsfonts}
\usepackage{amsbsy}
\usepackage{epsfig,afterpage}
\usepackage[hidelinks]{hyperref}
\usepackage[dvips]{psfrag}
\usepackage{color}     
\usepackage{overpic}
\usepackage{xcolor}
\usepackage{faktor}
\usepackage{caption}
\usepackage{bm}
\usepackage[pagewise]{lineno}

\usepackage{float}
\usepackage{tikz, wrapfig}
\usetikzlibrary{matrix,arrows,arrows.meta,bending,decorations.pathreplacing,decorations.pathmorphing,decorations.markings,fit,patterns,shapes,intersections,calc}
\tikzset{node distance=3cm, auto}
\allowdisplaybreaks

\newcommand{\R}{\ensuremath{\mathbb{R}}}

\newcommand{\N}{\ensuremath{\mathbb{N}}}

\newcommand{\Z}{\ensuremath{\mathbb{Z}}}

\newtheorem {mtheorem} {Theorem}

\newtheorem {proposition}  {Proposition}

\newtheorem {lemma}  {Lemma}
\newtheorem {definition}  {Definition}
\newtheorem {remark} {Remark}
\newtheorem {example} {Example}

\newtheorem {corollary}  {Corollary}

\begin{document}


\title{Symbolic dynamics of planar piecewise smooth vector fields}

\author[A. A. Antunes,  T. Carvalho] {Andr\'{e} do Amaral Antunes$^1$, Tiago Carvalho$^2$ }

\address{$^1$ IBILCE/UNESP, Zip Code 15054-000, S\~ao Jose do Rio Preto, S\~ao Paulo, Brazil}
\address{$^2$ FFCLRP-USP, Zip Code 14040-901, Ribeir\~ao Preto, S\~ao Paulo, Brazil}
\email{andre.antunes@unesp.br}
\email{tiagocarvalho@usp.br}




\begin{abstract}
	Recently, the theory concerning  piecewise smooth vector fields (PSVFs for short) have been undergoing important improvements. In fact, many results obtained do not have an analogous for smooth vector fields. For example, the chaoticity of planar PSVFs, which is impossible for the smooth ones. These differences are generated by the non-uniqueness of trajectory passing through a point. Inspired by the classical fact that one-dimensional discrete dynamic systems can produce chaotic behavior, we construct a conjugation between the shift map and  PSVFs. By means of the results obtained and the techniques employed, a new perspective  on the study of PSVFs is brought to light and, through already established results for discrete dynamic systems, we will be able to obtain results regarding PSVFs.
\end{abstract}

\maketitle

\section{Introduction}\label{intro}

In this paper we investigate some correlations between dynamical systems continuous  and discrete in time. In fact, several times, when working with continuous flows, one creates a single transformation of the space (or a subset of it) on  itself creating a discrete dynamical system. For example, when studying a  physical phenomenon and the measures of such can only be made in a discrete set of time; or when studying a Poincaré Map (first return map) that gives valuable information of a continuous  flow concerning stability and cyclicity. The inverse procedure  also is useful. For example, a biological phenomenon, where the number of individuals (cells, preys, etc) are natural numbers, can be modeled using an ordinary differential equation continuous in time (instead of a discrete function defined in the set $\N$ of natural number). In other cases, it is possible  to create a transformation of the flow domain (or a subset of it) on a set of discrete objects (for example, numbers either in $\Z$ or $\N$) and using the so called \textit{symbolic dynamics} to infer some important properties of that flow.

In parallel, there are a lot of applied problems that can be modeled using flows continuous in time but in such a way that the vector field involved is not smooth. This happens when the system is suddenly submitted to a \textit{on-off} change. For example, when we are studying the evolution of cancer cells (resp., HIV infected cells) in a patient submitted to an intermittent protocol of treatment, where chemotherapy (resp., antiretroviral) is administered in periodic periods (see  \cite{RMCG2019,Cancer-AMC-2019,Tang:2012}).  We will use the term \textit{Piecewise Smooth Vector Fields} (PSVFs, for short) to refer to these non-smooth vector fields.

One of the  most relevant characteristic about PSVFs is the non uniqueness of trajectory passing through a point, which means that if one takes a point, it is possible to  ask where its trajectory is going  after an amount of time $ t $, and we may have several different answers for that. But, if one takes a whole trajectory instead of a point, there is no ambiguity in that question, which indicates that we may have some more properties of the PSVFs when studying the space of orbits of that system.

The literature about PSVFs predicts the existence of chaotic behavior even in the planar (two-dimensional) case (see \cite{Ball-LuizTiago,BCEchaotic}). The proofs establish in the previous papers take into in count the classical definition of chaos (sensibility to the initial conditions $+$ topological transitivity). However, at the present paper, we introduce a fully new way of addressing this issue. In fact,  we induce a discrete dynamics in the space of orbits of the PSVFs and use symbolic dynamics to describe this system. Moreover, the results obtained can be useful, in the future, to answer another open questions, like:

\begin{itemize}
	\item What is the topological entropy of PSVFs?
	\item Is it possible the obtainment of Horse Shoes for (planar) PSVFs?
	\item Is it possible to determine a fractal (box) dimension for PSVFs?
\end{itemize} 
In fact, the deepness of the results and techniques presented in this work is expected to be revealed since a wide range of problems concerning PSVFs can be addressed by following classical scripts in symbolic dynamics, and vice versa.

\section{Main results}

Given a flow $\varphi(t,x)$ of a vector field $W$ defined in an open set $ \mathcal{U} $, the time-one map is the function $ T_{1}:\mathcal{U}\to\mathcal{U} $, $ T_{1}(x)=\varphi(1,x) $. The next result relates time-one maps of PSVFs (see Definition \ref{defi-time-one-map}) and (sub)shifts.

\begin{mtheorem}\label{teo 1} 
	\begin{itemize}
		
		\item[(i)] There exists a conjugacy between the time-one map of the PSVF 
		\begin{equation}\label{eq Z 2 simbolos}
				Z_{2}(x,y)=\left\{
				\begin{array}{l} 
					X_{2}(x,y)= \left(1,\frac{x}{2}-4x^3\right) ,\quad $for$ \quad y \geq 0 \\ 
					Y_{2}(x,y)= \left(-1,\frac{x}{2}-4x^3\right),\quad $for$ \quad y \leq 0
				\end{array}
				\right. ,
				\end{equation} 
		restricted to an invariant compact set $ \Lambda_2 \subset \R^{2}$ and the full shift of two symbols $ \sigma: \{0,1\}^{\Z}\to\{0,1\}^{\Z}  $
		
%
		
		\item[(ii)] For each $ k\geqslant 3 $, there exists a planar PSVF
	\begin{equation}\label{eq Z k simbolos}
	Z_{k}(x,y)=\left\{
	\begin{array}{l} 
	X_{k}(x,y)= \left(1, P_{k}'(x) \right) ,\quad $for$ \quad y \geq 0 \\ 
	Y_{k}(x,y)= \left(-1,P_{k}'(x)\right),\quad $for$ \quad y \leq 0
	\end{array}
	\right. ,
	\end{equation}where 
	\begin{equation}\label{polinomio Pk}
	P_{k}(x)=- \left(x+\frac{k-1}{2}\right) \left(x-\frac{k-1}{2}\right) \prod\limits_{i=1}^{k-1} \left(x-\left(i-\frac{k}{2}\right)\right)^2
	\end{equation}such that, restricted to an  invariant compact set $ \Lambda_{k} $, the time-one map is conjugated to a subshift of $ 2(k-1)$ symbols. 
		
%
%

		\item[(iii)] There exists a conjugacy between the time-one map of the PSVF 
		\begin{equation}\label{eq Z infinito}
		Z_{\infty}(x,y)=\left\{
		\begin{array}{l} 
		X_{\infty}(x,y)= (1,2\sin(2\pi x)) ,\quad $for$ \quad y \geq 0 \\ 
		Y_{\infty}(x,y)= (-1,2\sin(2\pi x)),\quad $for$ \quad y \leq 0
		\end{array}
		\right. ,
		\end{equation} 
		restricted to an invariant set $ \Lambda_{\infty} \subset \R^{2}$ and a subshift over an infinite alphabet $ \sigma: \Theta_{\infty}\to\Theta_{\infty} $. 

		
		\item[(iv)]  The return map of the PSVF
			\begin{equation}\label{Eq Campo feijao}
		Z(x,y)=\left\{
		\begin{array}{l} 
		X(x,y)= (1, -2x) ,\quad $for$ \quad y \geq 0 \\ 
		Y(x,y)= (-2,-4x^3+2x),\quad $for$ \quad y \leq 0
		\end{array}
		\right. ,
		\end{equation}restricted to an invariant compact set $ \Lambda$,  is conjugated to $\sigma:(0,1]^{\Z}\to(0,1]^{\Z} $.

		\end{itemize}

	\end{mtheorem}

In Section \ref{secao exemplos} is exhibited an example of PSVF whose time one map is conjugated to sub-shift of $4$-symbols.


%
%
%

In the next result we state that the canonical forms presented in items (i) to (iv) of Theorem \ref{teo 1} represent a bigger class of PSVFs conjugated to the (sub)shifts mentioned. Before announcing it, we indicate Definitions \ref{def folds} and \ref{def loops} where the concept of \textit{fold points} and \textit{homoclinic loops} are defined.

\begin{mtheorem}\label{teo 2}
	\begin{itemize}
		\item[(i)] The PSVF $Z_{2}$ of Theorem A, restrict to $\Lambda_2$, is $\Sigma$-equivalent to any PSVF presenting a $1$-homoclinic loop. 
		
		\item[(ii)] 	The PSVF $Z_k$ of Theorem A, restrict to $\Lambda_k$, is  $\Sigma$-equivalent to any PSVF presenting a $(k-1)$-homoclinic loop.
		
		\item[(iii)] The PSVF $Z_\infty$ of Theorem A, restrict to $\Lambda_\infty$, is  $\Sigma$-equivalent to any PSVF presenting a $\infty$-homoclinic loop.
		
		\item[(iv)]  The PSVF $Z$ of Theorem A, restrict to $\Lambda$, is $\Sigma$-equivalent to any PSVF $\widetilde{Z}$ presenting a compact region $\widetilde{\Lambda}$ bounded by a trajectory of $\widetilde{Z}$ passing through a invisible-visible  two-fold $\widetilde{p}$. Moreover, except for $\widetilde{p}$, the PSVF $\widetilde{Z}$ has just more two invisible tangential singularities.
	\end{itemize}
\end{mtheorem}

We stress that the restriction to the pictured sets of both PSVFs in Figure \ref{fig-equivalencias}, according to Theorem B item (i), are $\Sigma$-equivalents  to $Z_2$ restricted to $\Lambda_2$. That means that our result is not restrict to general PSVFs with invariant sets with the shape of a ``figure eight lying''.  The same holds for $Z_k$, with $k \in \{ 3,4,\hdots \} \cup \{ \infty\}$.
	\begin{figure}[h]
	\includegraphics[width=0.85\linewidth]{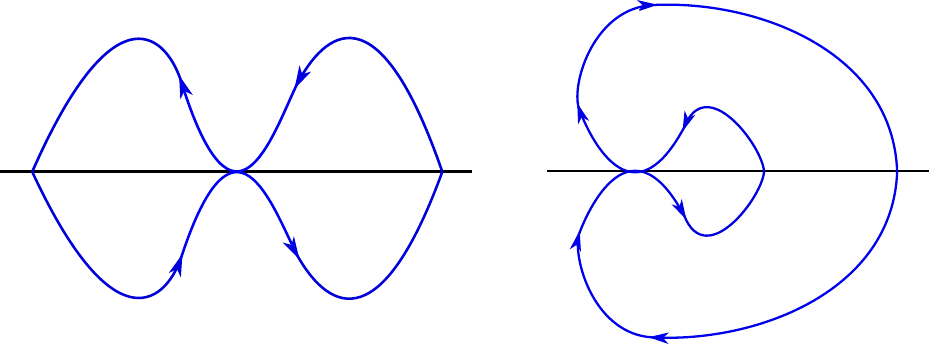}
	\caption{Disctint kinds of PSVFs whose restriction to an invariant set is $\Sigma$-equivalent to a restriction of $Z_2$ to $\Lambda_2$. Also, it is important to highlight that these PSVFs could have non-polynomial expressions. An analogous remark could be done for the other cases in Theorem A.}
	\label{fig-equivalencias}
\end{figure}

As an immediate consequence of the previous results we have that:

\begin{corollary}
	The time-one maps of PSVFs presented in Theorems \ref{teo 1} and \ref{teo 2} have periodic points of any period and are topologically mixing (see Definition \ref{def-mixing}).
\end{corollary}

\begin{corollary}
	All PSVFs presented in items (i) to (iv) of Theorems \ref{teo 1} and \ref{teo 2} are chaotic.
\end{corollary}


\section{General theory concerning PSVFs, Shifts and diffeomorphisms}


%
\subsection{Piecewise Smooth Vector Fields}

\

Consider a codimension one manifold $\Sigma$ of $\R^n$ given by
$\Sigma =f^{-1}(0),$ where $f:\R^n \rightarrow \R$ is a smooth function
having $0\in \R$ as a regular value (i.e. $\nabla f(p)\neq 0$, for
any $p\in f^{-1}({0}))$. Assume that $\Sigma$ is an embedded submanifold of $\R^n$.  We call $\Sigma$ the \textit{switching
	manifold} that is the separating boundary of the regions
$\Sigma^+=\{q\in \R^n \, | \, f(q) \geq 0\}$ and $\Sigma^-=\{q \in \R^n\,
| \, f(q)\leq 0\}$.

Designate by $\mathfrak{X}^r$ the space of C$^r$-vector fields on
$V\subset\R^n$ endowed with the C$^r$-topology, with $r \geq 1$
large enough for our purposes. Call $\Omega^r$ the space of piecewise smooth vector fields (PSVFs for short)
$Z: \R^n \rightarrow \R ^{n}$ such that
\begin{equation}\label{eq Z}
Z(q)=\left\{\begin{array}{l} X(q),\quad $for$ \quad q \in
\Sigma^+,\\ Y(q),\quad $for$ \quad q \in \Sigma^-,
\end{array}\right.
\end{equation}
where $X=(X_1,\ldots,X_n) , Y = (Y_1,\ldots,Y_n) \in \mathfrak{X}^r$. We endow
$\Omega^r$ with the product topology. 
%

For practical purposes, the contact between the smooth vector field
$X$ and the switching manifold $\Sigma = f^{-1}(0)$ is characterized
by the expression $Xf(p)=\left\langle \nabla f(p),
X(p)\right\rangle$  where $\langle . , . \rangle$ is the usual inner product in $\R^n$.

The basic results of differential equations in this context were stated by Filippov (see \cite{Fi}). We can distinguish on $\Sigma$ the following regions:

$\bullet$ Crossing Region: $\Sigma^c=\{ p \in \Sigma \, | \, Xf(p)\cdot Yf(p)> 0 \}$.
Moreover, we denote $\Sigma^{c+}= \{ p \in \Sigma \, | \,
Xf(p)>0, Yf(p)>0 \}$ and $\Sigma^{c-} = \{ p \in \Sigma \, | \,
Xf(p)<0,Yf(p)<0 \}$.

$\bullet$ Sliding Region: $\Sigma^{s}= \{ p \in \Sigma \, | \, Xf(p)<0,
Yf(p)>0 \}$.

$\bullet$ Escaping Region: $\Sigma^{e}= \{ p \in \Sigma \, | \, Xf(p)>0
,Yf(p)<0\}$.

\begin{definition}\label{def folds}In the case $Xf(p)= 0$ we say that $p$ is a \textbf{tangential singularity} of $X$. A tangential singularity $p \in
\Sigma$ is a \textbf{fold point} of $X$ if $Xf(p)=0$ but $X^{2}f(p)\neq0$, where  $X^if(p)=\left\langle \nabla X^{i-1}f(p), X(p)\right\rangle$ for $i\geq 2$. Moreover, $p\in\Sigma$ is a \textbf{visible} (respectively {\bf invisible}) fold point of $X$ if $Xf(p)=0$ and $X^{2}f(p)> 0$ (respectively $X^{2}f(p)< 0$). A point $ p\in\Sigma $ is a two-fold, if it is a fold point for both $ X $ and $ Y $, and it is visible-visible if visible for both (respectively invisible-visible and invisible-invisible).\end{definition}

In addition, a tangential singularity $q$ is \textit{singular} if $q$ is a invisible tangency for both $X$ and $Y$.  On the other hand, a tangential singularity $q$ is \textit{regular} if it is not singular. 

\begin{definition}\label{definicao campo deslizante tangencial} Given a point $p \in \Sigma^s\cup\Sigma^e \subset \Sigma$, we define the \textbf{sliding vector field}  at $p$ as the vector field $Z^{T}(p)=m-p$ with
	$m$ being the point of the segment joining $p+X(p)$ and $p+Y(p)$
	such that $m-p$ is tangent to $\Sigma$. 
\end{definition}


In the plane, the sliding vector field is given by the expression
\begin{equation}\label{expfilipov}
Z^{T}(p)=\frac{Yf(p)X(p)-Xf(p)Y(p)}{Yf(p)-Xf(p)}.
\end{equation}
Moreover, the sliding vector field can be extend to $\overline{\Sigma^e}\cup\overline{\Sigma^s}$.

Now we establish the classical convention on the trajectories of orbit-solutions of a PSVF.
\begin{definition}\label{definicao trajetorias}
	The \textbf{local trajectory (orbit)} $\phi_{Z}(t,p)$ of a PSVF
	given by \eqref{eq Z} through $p\in V$ is defined as follows:
	\begin{itemize}
		\item[(i)] For $p \in \Sigma^+ \backslash \Sigma$ and $p \in \Sigma^{-} \backslash \Sigma$ the trajectory is
		given by $\phi_{Z}(t,p)=\phi_{X}(t,p)$ and
		$\phi_{Z}(t,p)=\phi_{Y}(t,p)$ respectively, where $t\in I$.
		
		\item[(ii)] For $p \in \Sigma^{c+}$ and  taking the
		origin of time at $p$, the trajectory is defined as
		$\phi_{Z}(t,p)=\phi_{Y}(t,p)$ for $t\in I\cap \{t\leq 0\}$ and
		$\phi_{Z}(t,p)=\phi_{X}(t,p)$ for $t\in I\cap \{t\geq 0\}$. For
		the case $p \in \Sigma^{c-}$  the definition is the same
		reversing time.

		\item[(iii)] For $p \in \Sigma^e$ and  taking the
		origin of time at $p$, the trajectory is defined as
		$\phi_{Z}(t,p)=\phi_{Z^{T}}(t,p)$ for $t\in I\cap \{t\leq 0\}$ and
		$\phi_{Z}(t,p)$ is either $\phi_{X}(t,p)$ or $\phi_{Y}(t,p)$ or $\phi_{Z^{T}}(t,p)$ for $t\in I\cap \{t\geq 0\}$. For
		$p \in \Sigma^s$ the definition is the same
		reversing time.

		\item[(iv)] For $p$ a regular tangency point and  taking the
		origin of time at $p$, the trajectory is defined as
		$\phi_{Z}(t,p)=\phi_{1}(t,p)$ for $t\in I\cap \{t\leq 0\}$ and
		$\phi_{Z}(t,p)=\phi_{2}(t,p)$ for $t\in I\cap \{t\geq 0\}$, where each $\phi_{1},\phi_{2}$ is either $\phi_{X}$ or $\phi_{Y}$ or $\phi_{Z^{T}}$.
		
		\item[(v)]For $p$ a singular tangency point, $\phi_{Z}(t,p)=p$ for all $t \in \R$.
	\end{itemize}
\end{definition}

\begin{definition}\label{definicao trajetoria global}
	The \textbf{ global trajectory} (\textbf{orbit}) $\Gamma_Z (t,p_0)$ of $Z$ passing through $p_0$ is a union $$\Gamma_Z (t,p_0) = \bigcup_{i\in\mathbb{Z}}\{ \sigma_{i}(t,p_i) : t_i \leq t\leq t_{i+1} \}$$ of preserving-orientation local trajectories $\sigma_{i}(t,p_i)$ satisfying $\sigma_{i}(t_{i+1},p_i)=\sigma_{i+1}(t_{i+1},p_{i+1})=p_{i+1}$ and $t_i\rightarrow\pm\infty$ as $i\rightarrow\pm\infty$.
\end{definition}

\begin{definition}\label{def loops}
		 	A \bm{$k$}\textbf{-homoclinic loop} of a planar PSVF is a global trajectory of $Z=(X,Y)$ presenting $k$ distinct visible-visible two-fold singularities $p_1, \hdots, p_k$ in such a way that, after passes through $p_i$, the trajectory reaches $\Sigma$ either in $p_{i-1}$ or $p_{i+1}$ when $i=2,\hdots,k-1$. When $i=1$ (resp., $i=k$),  after passes through $p_i$, the trajectory reaches $\Sigma$  either in a sewing point or $p_{i+1}$ (resp., $p_{i-1}$).
		 	
		 	Moreover, an $\mathbf{\infty}$\textbf{-homoclinic loop} of a planar PSVF is a global trajectory of $Z$ presenting $\infty$ visible-visible two-fold singularities $p_1, p_2, \hdots$ in such a way that, after passes through $p_i$, the trajectory reaches $\Sigma$ either in $p_{i-1}$ or $p_{i+1}$. 
		 	
		 	The homoclinic loops $\Gamma$ defined above are \textbf{regular} if  $X(p)$ and $Y(p)$ point to opposite direction, for all two-fold singularity $p \in \Gamma$ (remember that $X(p)$ and $Y(p)$  are parallel). Otherwise, the homoclinic loop is \textbf{singular}.  Figures \ref{fig-equivalencias} and \ref{three-symbols} illustrate some regular $k$-homoclinic loops.
\end{definition}

Another important definition is the concept of equivalence between two PSVFs.

\begin{definition}\label{definicao sigma-equivalencia}
	Two PSVFs $Z=(X,Y), \,
	\widetilde{Z}=(\widetilde{X},\widetilde{Y}) \in \Omega^r$, defined in
 $U, \, \widetilde{U}$ respectively and with switching manifold $\Sigma$ and $\widetilde{\Sigma}$
	are \textbf{$\mathbf{\Sigma}$-equivalent} if there exists an
	orientation preserving homeomorphism $h: U \rightarrow
	\widetilde{U}$ that sends 
	$U\cap\Sigma$ to
	$\widetilde{U}\cap\widetilde{\Sigma}$, the orbits of  $X$ restricted to
	$U\cap\Sigma^+$ to the orbits of $\widetilde{X}$ restricted to
	$\widetilde{U}\cap\widetilde{\Sigma}^+$,  the orbits of  $Y$ restricted to
	$U\cap\Sigma^-$ to the orbits of $\widetilde{Y}$ restricted to
	$\widetilde{U}\cap\widetilde{\Sigma}^-$ and the orbits of  $Z^T$ restricted to
	$\Sigma$ to the orbits of $\widetilde{Z}^T$ restricted to
	$\widetilde{\Sigma}$.
\end{definition}

\subsection{Symbolic Dynamics}

\

Consider a set $\mathcal{A}_{k} $ with $ k $ elements (say $ \mathcal{A}_{k}= \{0,1, \dots, k-1\} $) with the discrete topology. Now, consider $ \mathcal{A}_{k}^{\Z} $,  i.e., all the sequences $ x=(x_j)_{j\in\Z} $, with $ x_j\in \mathcal{A}_{k},$ for all $ j $ and the product topology of all discrete topologies.


\begin{definition}
	Let $ x=(x_{j})_{j\in\Z} $ and $ y=(y_{j})_{j\in\Z} $ two elements of $ \mathcal{A}_{k}^{\Z} $. Define $ d:\mathcal{A}_{k}^{\Z}\times\mathcal{A}_{k}^{\Z}\to\R $ by:
	\[ d(x,y)=\sum_{j\in\Z}\frac{|x_{j}-y_{j}|}{2^{|j|}} \]
\end{definition}

\begin{definition}
	Define $ \sigma_{k}:\mathcal{A}_{k}^{\Z}\to\mathcal{A}_{k}^{\Z} $ given by $ \sigma((a_j))=(b_j) $, where $ b_j=a_{j+1} $. The map $ \sigma $ is called \textbf{two-sided full shift} and the discrete flow $ (\mathcal{A}_{k}^{\Z},\sigma) $ is called \textbf{symbolic flow} or \textbf{shift system}.
\end{definition}

\begin{proposition}\label{metrica-shift}
	The following holds:
	\begin{itemize}
		\item[(i)] The function $ d $ defined above is a metric on $ \mathcal{A}_{k}^{\Z} $ and induces the product topology;
		\item[(ii)] The space $ \mathcal{A}_{k}^{\Z} $ is a compact Hausdorff space;
		\item[(iii)] $\sigma_{k}$ is a homeomorphism.
	\end{itemize}
\end{proposition}


\begin{proof}	In \cite{Katok95}, p. 48. the authors assert that the results in items (i) and (ii) of this proposition are true, but without a proof. So, for completeness, we proceed the proof below.
	
	(i) The convergence of $ d $ follows from comparison test with the geometric series and the properties for it to be a metric follows easily from absolute value in the numerator. To see that the topology is the same, consider $ \mathcal{U} $ an open basic set and $ x\in\mathcal{U} $. Wlog, one can assume that $ \mathcal{U}=\prod\limits_{j<-\alpha}\mathcal{A}\times\{x_{-\alpha}\}\times\dots\times\{x_{\alpha}\}\times\prod\limits_{j>\alpha}\mathcal{A} $, for some $ \alpha\in\N $. That is, any $ y\in\mathcal{U} $ coincides with $ x $ in all entries between $ -\alpha $ and $ \alpha $.
	
	Let $ 0<r<\frac{1}{2^{\alpha}} $ and consider the open ball $ B=B(x,r) $. Let $ y\in B $. Suppose there exists $ \beta\in\Z $, $ -\alpha\leqslant\beta\leqslant\alpha $ such that $ y_{\beta}\neq x_{\beta} $. Then 
	
	\[ d(x,y)=\sum\limits_{j\in\Z}\frac{|x_{j}-y_{j}|}{2^{|j|}}\geqslant\frac{1}{2^{|\beta|}}\geqslant\frac{1}{2^{\alpha}}>r \]
	
	This is a contradiction, since $ y\in B $. Then $ y\in\mathcal{U} $.
	
	On the other hand, let $ x\in\mathcal{A}_{k}^{\Z} $ and $ B=B(x,r) $ an open ball. There exists $ \alpha\in\N $, such that $ \frac{k-1}{2^{\alpha-1}}< r $. Let $ \mathcal{U} $ be an basic open set as before. We have $ x\in\mathcal{U} $ and let $ y\in\mathcal{U} $. Then:
	
	\[ d(x,y)=\sum\limits_{j\in\Z}\frac{|x_{j}-y_{j}|}{2^{|j|}}= \sum\limits_{\substack{j\in\Z\\|j|>\alpha}}\frac{|x_{j}-y_{j}|}{2^{|j|}}\leqslant 2(k-1)\sum\limits_{j=\alpha+1}^{\infty}\frac{1}{2^{j}}= \frac{k-1}{2^{\alpha-1}}<r \]

	(ii) Each copy of $ \mathcal{A} $ is compact, then, by Tychonoff's Theorem, $ \mathcal{A}_{k}^{\Z} $ is compact. Moreover, it is a metric space, hence it is Hausdorff  
	
	(iii) It is enough to note that $ \sigma_{k} $ takes basic open set to basic open set.
	
\end{proof}

When it is clear by the context, the subscript $ k $ may be suppressed in the sequel.

\begin{definition}
	Let $ K\subset\mathcal{A}^{\Z} $. We say $ (K,\sigma) $ is a \emph{subshift} if $ K $ is closed and invariant for $ \sigma $.
\end{definition}

%
%
%
%

There are some works (see \cite{goncalves_sobottka_starling_2017} and \cite{Ott2014}) considering one and two sided shifts over enumerable alphabets. To the best of our knowledge, there is not a classic way to work with them, and when considering the full shift, it is difficult to define a topology over it in order to make it metrizable. 

In this work, we consider a shift over $ \Z $ presenting some restrictions over the sequences which are allowed to occur. These features allow us to define a metric over this space.

\begin{definition}\label{def-theta-infinito}
	Let $ \Theta_{\infty} \subset \Z^{\Z}$ be the set of bi-infinite sequences with integer entries $ (x_j)_{j\in\Z} $, such that 
	the difference between two consecutive entries is  at most 2  . That is:
	\[ \Theta_{\infty}=\{(x_{j})_{j\in\Z}\,\,|\,\,x_{j}\in\Z \text{ and } |x_{j+1}-x_{j}|\leqslant 2, \forall j\in\Z\} \]
\end{definition}

We can define over $ \Theta_{\infty} $ the same metric given before.

\begin{definition}
	Consider $ x=(x_{j})_{j\in\Z} $ and $ y=(y_{j})_{j\in\Z} $ two elements of $ \Theta_{\infty} $. Define $ d:\Theta_{\infty}\times\Theta_{\infty}\to\R $ by:
	\[ d(x,y)=\sum_{j\in\Z}\frac{|x_{j}-y_{j}|}{2^{|j|}} \]
\end{definition}

\begin{proposition}
	The function $ d $ defined above is a metric over $ \Theta_{\infty} $.
\end{proposition}
\begin{proof}
	Once proven the convergence of $ d $, the properties for it to be a metric follows straightforward from the absolute value on numerator of the series.
	
	Let $ x=(x_{j})_{j\in\Z}, y=(y_{j})_{j\in\Z} \in\Theta_{\infty}$, then for all $ j\in\Z $:
	\[ |x_{j}-y_{j}|\leqslant|x_{j}-x_{j-1}|+|x_{j-1}-y_{j-1}|+|y_{j-1}-y_{j}|\leqslant |x_{j-1}-y_{j-1}|+4 \]
	
	Analogously, one can show: $ |x_{j}-y_{j}|\leqslant |x_{j+1}-y_{j+1}|+4 $. And, recursively, we obtain: 
	\[ \forall j\in\Z: |x_{j}-y_{j}|\leqslant |x_0-y_0|+4|j|. \] 
	
	Now we are in conditions to prove the convergence of $ d $. Given $ N\in\N $:
	
	\begin{align*}
	\sum_{j=-N}^{N}\frac{|x_{j}-y_{j}|}{2^{|j|}} 
	&= |x_0-y_0|+ \sum_{j=1}^{N}\frac{|x_{j}-y_{j}|}{2^{j}} +\sum_{j=1}^{N}\frac{|x_{-j}-y_{-j}|}{2^{j}}\leqslant\\
	&\leqslant |x_0-y_0|+ \sum_{j=1}^{N}\frac{|x_{0}-y_{0}|+4j}{2^{j}} +\sum_{j=1}^{N}\frac{|x_{0}-y_{0}|+4j}{2^{j}}\\
	&\leqslant |x_0-y_0| + 2|x_{0}-y_{0}|\sum_{j=1}^{N}\frac{1}{2^{j}}+8\sum_{j=1}^{N}\frac{j}{2^{j}}\leqslant\\
	&\leqslant 3|x_0-y_0|+16<\infty
	\end{align*}
	
	Making $ N\to\infty $, we have that $ d $ converges.
	
\end{proof}

Two important properties of a dynamical system $ (X,f) $ are defined below:

\begin{definition}\label{def-mixing}
	Let $ (X,f) $ be a dynamical system. We say that $ f $ is \emph{topologically transitive} if for every open sets $ \mathcal{U}, \mathcal{V} $, there exists $ n $, such that $ f^{n}(\mathcal{U})\cap\mathcal{V}\neq\emptyset $. And $ f $ is \emph{topologically mixing} if for every open sets $ \mathcal{U}, \mathcal{V} $, there exists $ n_{0} $, such that $ f^{n}(\mathcal{U})\cap\mathcal{V}\neq\emptyset $ for all $ n\geqslant n_{0} $.
\end{definition}

\begin{proposition}
	If $ f $ is topologically mixing, then $ f $ is topologically transitive.
\end{proposition}

\begin{proposition}
	The two-sided shift $ \sigma_{k} $ is topologically mixing.
\end{proposition}

\begin{definition}
	The Hausdorff distance between the sets $A$ and $B$ is given by
	\[ d_H(A,B)=max \{ \sup_{x \in A} \inf_{y \in B}d(x,y), \sup_{y \in B} \inf_{x \in A} d(x,y)  \} \]
\end{definition}

\section{Proof of Theorem 1}

For each natural number $ k\geqslant 2 $, let:

\begin{equation}
P_{k}(x)=- \left(x+\frac{k-1}{2}\right) \left(x-\frac{k-1}{2}\right) \prod\limits_{i=1}^{k-1} \left(x-\left(i-\frac{k}{2}\right)\right)^2
\end{equation}

and 

\begin{equation}\label{polinomio Pinf}
P_{\infty}(x)= 1-\cos(2\pi x)
\end{equation}

\begin{lemma}\label{lema-raizes}
	Let $ P_{k} $ and $ P_{\infty} $ as defined above. Then:
	\begin{itemize}
		\item[(i)] $P_k$ has $2 k$ roots, being $2$ simple roots at $ r_0=\frac{1-k}{2}$ and $r_1=\frac{k-1}{2}$ and $k-1$ roots of multiplicity two at $p_j=j-\frac{k}{2}$ for $j=1,\dots,k-1$. Moreover, $P_{k}'(r_0)>0$, $P_{k}'(r_1)<0$, $P_{k}'(p_j)=0$ and $P_{k}''(p_j)>0$ for every $ j=1,\dots,k-1 $.
		\item[(ii)] $P_{\infty}$ has infinitely many roots placed at  $p_j=j$ for $j \in \Z$. Moreover $P_{\infty}'(p_j)=0$ and  $P_{\infty}''(p_j)>0$. As consequence,  each $p_j$ is a root of multiplicity two.
	\end{itemize}

\end{lemma}

\begin{proof}
	The placement of those roots follows easily from the factored form of $ P_{k} $ and the expression for $ P_{\infty} $ presented in (\ref{polinomio Pk}) and (\ref{polinomio Pinf}).
	
	For the derivatives part, if $ k<\infty $, put $ g(x)=\left(x+\frac{k-1}{2}\right) \left(x-\frac{k-1}{2}\right) $ and $ h(x)=\prod\limits_{i=1}^{k-1} \left(x-\left(i-\frac{k}{2}\right)\right)^2 $, then $ P_{k}(x)=-g(x)h(x) $ and
	\[ P_{k}'(x)=-g(x)h'(x)-g'(x)h(x)=-g(x)h'(x)-2xh(x) \Rightarrow \]
	\[\Rightarrow P_{k}'(r_{0})=-2r_{0}h(r_{0})>0 \text{ and } P_{k}'(r_{1})=-2r_{1}h(r_{1})<0\]

	Since $ p_{j} $ are roots of multiplicity two of $ h(x) $, then $ h(x)=h'(x)=0 $, which implies $ P_{k}'(p_{j})= 0 $ and:
	\[ P_{k}''(p_{j})=-g(p_{j})h''(p_{j}) \]
	
	Now, we have
	\[ h''(p_{j})=2\prod\limits_{\substack{i=1\\i\neq j}}^{k-1} \left(x-\left(i-\frac{k}{2}\right)\right)^2 >0 \]
	
	And, if $  |x| <\frac{k-1}{2}  $:  
	\[ g(x)=\left(x-\frac{1-k}{2}\right)\left(x-\frac{k-1}{2}\right)< 0 \] 
	
	Thus, $ P_{k}''(p_{j})>0 $.

	For $ P_{\infty} $, we have:
	\[ P_{ \infty}'(x)=2\sin(2\pi x) \Rightarrow P_{ \infty}'(j)=0, \forall j\in\Z  \]
	and
	\[ P_{ \infty}''(x)=4\cos(2\pi x) \Rightarrow P_{ \infty}''(j)=1>0, \forall j\in\Z \]

\end{proof}

Now, for every $ k\geqslant 2 $ or $ k=\infty $, define:

\begin{equation}\label{eq Z k simbolos}
	Z_{k}(x,y)=\left\{
	\begin{array}{l} 
	X_{k}(x,y)= \left(1, P_{k}'(x) \right) ,\quad $for$ \quad y \geq 0 \\ 
	Y_{k}(x,y)= \left(-1,P_{k}'(x)\right),\quad $for$ \quad y \leq 0,
	\end{array}
	\right.
\end{equation}

\begin{lemma}[Tangencies of $ Z_{k} $]\label{lema tangencias}
	Let $ Z_{k} $ be as above, then the following holds:
	\begin{itemize}
		\item If $ k<\infty $ then $(r_0,0)$ and $(r_1,0)$ are crossing points of $ Z_{k} $;
		\item The points $(p_j,0)$ are visible-visible two folds of $Z_{k}$, $j=1,\dots,k-1$ (or $ j\in\Z $, if $ k=\infty $);
		\item $\Sigma^s \cup \Sigma^e = \emptyset$.
	\end{itemize}

\end{lemma}

\begin{proof}
	The switching manifold is given by $ \Sigma=f^{-1}(0) $, where $ f(x,y)=y $, then $ X_{k}f(r_{0},0).Y_{k}f(r_{0},0)=(P_{k}'(r_0))^{2}>0 $. Analogous for $ (r_{1},0) $.
	
	From Lemma \ref{lema tangencias}, $ X_{k}f(p_{j},0)=Y_{k}f(p_{j},0)=P_{k}'(p_{j})=0 $ and $ X_{k}^{2}f(p_{j},0)=Y_{k}^{2}f(p_{j},0)=P_{k}''(p_{j})>0 $.
\end{proof}

For each $ k<\infty $, let
\[\gamma_{k}^{X}= \{ (x,P_{k}(x)) \,\, | \,\, x\in[r_0,r_1]\} \mbox{ and } \gamma_{k}^{Y}=\{ (x,-P_{k}(x)) \,\, | \,\, x\in[r_0,r_1]\} \]
and, for $ k=\infty $,
\[ \gamma_{\infty}^{X}= \{ (x,P_{\infty}(x)) \,\, | \,\, x\in\R\} \mbox{ and } \gamma_{\infty}^{Y}=\{ (x,-P_{\infty}(x)) \,\, | \,\, x\in\R\}. \] 
Define $\Lambda_{k} = \gamma_{k}^{X} \cup \gamma_{k}^{Y} $ for both $ k<\infty $ or $ k=\infty $.

\begin{proposition}
	$\Lambda_{k}$ is an invariant set for $ Z_{k} $.
\end{proposition}
\begin{proof} 
	Note that $\gamma_{k}^{X} $ and $ \gamma_{k}^{Y} $ are integral curves of $X_{k}$ and $Y_{k}$, respectively. Moreover, $\gamma_{k}^{X}$ and $\gamma_{k}^{Y}$ coincide at $(p_j,0),$ for all $ j$ and at $ (r_0,0), (r_1,0) $ (in the case $ k<\infty $). By  Definitions \ref{definicao trajetorias}, \ref{definicao trajetoria global} and the description of these points given by Lemma \ref{lema tangencias} we obtain that $\Lambda_{k}$ is invariant.
\end{proof}

Our main goal is to prove the conjugacy between the time-one map of the fields $ Z_{k} $ restricted to $ \Lambda_{k} $ and a two-sided shift space. But, since a PSVF does not give uniqueness of trajectory through a point, a map $ T_{1}:\Lambda_{k}\to\Lambda_{k}, T_{1}(x)=\varphi(1,x) $, where $ \varphi $ is the flow, is not well-defined, for it may have more than one image (depending on the flow chosen). One way of avoiding this is to work with the space of all possible trajectories, so $ T_{1} $ will be well-defined. This strategy is inspired (and somewhat similar) to the one used in discrete dynamics when working with a non-invertible map, and constructing the \emph{inverse limit} of the map (see \cite{Katok03}). 

In that spirit, we will consider the set $\Omega_{k}=\{\gamma$ global trajectory of $Z_{k} \, | \, \gamma(0)\in \Lambda_{k}\}$ throughout this paper and we re-define the \emph{time-one map} as follows:

\begin{definition}\label{defi-time-one-map}
	We will call \emph{time-one map} the function $ T_1:\Omega_{k}\to\Omega_{k} $, given by $ T_1(\gamma)(.)=\gamma(.+1)$. 
\end{definition}

\begin{proposition}\label{flight-time}
	For any $ k<\infty $ (or $ k=\infty $), let $ \gamma \in\Omega_{k}$ be a global trajectory, then for all $ t\in\R $, there exists unique $ t^*\in[t,t+1) $ such that $ \gamma(t^*)\in\{(p_{j},0)\}_{j=1}^{k-1} $ (respectively, $ \gamma(t^*)\in\{(p_{j},0)\}_{j\in\Z} $).  
\end{proposition}

\begin{proof} Follows immediately from the expression of $ Z_{k} $ and Lemmas \ref{lema-raizes} e \ref{lema tangencias}.
	%
	%
	%
	%
	%
	
\end{proof}


The region $ \Lambda_{k} $ can be partitioned into arcs that goes from $ p_{j} $ to the adjacent ones ($ p_{j+1} $ and $ p_{j-1} $) or to itself (only when $ k<\infty $). So, consider $ k<\infty $ to be fixed and let $ I_0=\{(x,P_{k}(x)), x\in [r_0,p_1)\}\cup \{(x,-P_{k}(x)), x\in [r_0,p_1)\} $, i.e., the arc from $ p_1 $ to itself passing through $ r_0 $. For any $ j=1,\dots,k-2 $, let $ I_{2j-1}=\{(x,P_{k}(x)), x\in (p_{1},p_{2})\} $ and $ I_{2j}=\{(x,-P_{k}(x)), x\in (p_{1},p_{2})\} $, i.e, the arcs from $ p_j $ to $ p_{j+1} $ and from $ p_{j+1} $ to $ p_{j} $, respectively. And, $ I_{2k-3}=\{(x,P_{k}(x)), x\in (p_{k-1},r_{1}]\}\cup \{(x,-P_{k}(x)), x\in (p_{k-1},r_{1}]\} $. In short, we enumerate these arcs top to bottom, left to right (see Figure \ref{three-symbols}).

For the case $ k=\infty $ we will separate $ \Lambda_{\infty} $ into arcs, in the same way as we done for the  finite case: let $ I_{2j}=\{(x,P_{\infty}(x))\,\,|\,\,j<x<j+1\} $ and $ I_{2j+1}=\{(x,-P_{\infty}(x))\,\,|\,\,j<x<j+1\} $.

\begin{definition} 
	Let $ s:\Omega_{k}\to\Theta_{2k-2} $, given by $ s(\gamma)=(s_j(\gamma))_{j\in\Z} $ where:
	
	\begin{equation*}
	s_j(\gamma)=\left\{
	\begin{array}{l} 
	n,\quad $if$ \quad \gamma(j) \in I_n \\ 
	m,\quad $if$ \quad \gamma(j)  \in \{(p_{l},0)\} \mbox{ and } \gamma(j+1/2)\in I_m 
	\end{array}\right.
	\end{equation*}
	
	The sequence $ s(\gamma) $ is called the \emph{itinerary} of $ \gamma $.
\end{definition}

For $ k<\infty $, it is clear that $s$ is well-defined. For $ k=\infty $ we have to analyse if $ s(\gamma) \in\Theta_{\infty} $. From Definitions \ref{definicao trajetorias} and \ref{definicao trajetoria global}, of local and global trajectories, we see that if $ \gamma $ goes through some compartment $ I_{2l} $, then the next one must be the one below it ($ I_{2l+1} $) or the one to its right ($ I_{2l+2} $) and if it goes through $ I_{2l+1} $ then the following one must be above it ($ I_{2l} $) or the one to its left ($ I_{2l-1} $). In any case, $ 0<|s_{j}(\gamma)-s_{j+1}(\gamma)|\leqslant 2 $, so $ s(\gamma)\in\Theta_{\infty} $. From Proposition \ref{flight-time}, $ (s_{j}(\gamma))_{j\in\Z} $ encodes every compartment $ I_{j} $ that the trajectory $ \gamma $ visits in positive and negative time.

Note that, given $ \gamma\in\Omega_k $, there exists an infinity amount of different trajectories with the same itinerary of $ \gamma $, simply by changing the initial condition of it, without modifying the compartment it is located. To avoid such situation, we will consider two trajectories with the same itinerary to be equivalent. That is, we consider the equivalence relation:

\begin{definition}
	Let $ \gamma_1,\gamma_2 \in\Omega_{k}$. We say $ \gamma_1\sim\gamma_2 $ if and only if $ s(\gamma_1)=s(\gamma_2) $. Denote $ \overline{\Omega}_{k}=\faktor{\Omega_{k}}{\sim} $.
\end{definition}

\begin{remark}
	Observe that, given $ \overline{\gamma}\in\overline{\Omega}_{k} $, there exists a representative $ \gamma^*\in \overline{\gamma} $ such that $ \gamma^*(0)\in\{(p_{j},0), l=1,\dots,k-1\} $, because if $ \gamma(0)\in\{(p_{j},0), l=1,\dots,k-1\} $, simply take $ \gamma^*=\gamma $, if not, by Proposition \ref{flight-time} there exists unique $ -1<t^*<0 $ such that $ \gamma(t^*)\in\{(p_{j},0), l=1,\dots,k-1\} $. Moreover, if $ s(\gamma)=(s_{j})_{j\in\Z} $, then $ \gamma((t^*+j,t^*+j+1)) = I_{s_{j}} $. Which implies that $\gamma^*((j,j+1))= I_{s_{j}} $ and, consequently, $ \gamma^*(j+\frac{1}{2})\in I_{s_{j}} $. Then $ s(\gamma^*)=s(\gamma) $.
\end{remark}

In the following we construct a metric on the space $ \overline{\Omega}_{k} $, but for this, we use the Hausdorff distance  between two closed sets that is defined below:

\begin{definition}
	The Hausdorff distance between the sets $A$ and $B$ is given by
	\[ d_H(A,B)=max \{ \sup_{x \in A} \inf_{y \in B}d(x,y), \sup_{y \in B} \inf_{x \in A} d(x,y)  \} \]
\end{definition}

\begin{definition}
	Now, define $ \rho_{k}:\overline{\Omega}_{k}\times\overline{\Omega}_{k}\to\R $, by  
	\[ \rho_{k}(\overline{\gamma}_1,\overline{\gamma}_2)= \sum\limits_{i=-\infty}^{\infty}\frac{d_i(\overline{\gamma_1},\overline{\gamma_2})}{2^{|i|}} \]
	where $d_i(\overline{\gamma_1},\overline{\gamma_2})=d_H(\gamma_1^{*}([i,i+1]),\gamma_2^{*}([i,i+1]))$, $ d_H$ is the Hausdorff distance and $ \gamma_1^*, \gamma_2^* $ are those representatives given in the previous remark.
	
\end{definition}

In order to simplify notation, in the sequel we only refer to $ \gamma \in \overline{\Omega}_{k}$, meaning the equivalence class $ \overline{\gamma} $ with the representative $ \gamma^* $.


\begin{proposition}\label{prop-omega-metric}
	$ (\overline{\Omega}_{k},\rho_{k}) $ is a metric space.
\end{proposition}

\begin{proof}
	Let us show that $ \rho_{k} $ is well-defined, i.e., the series in the previous definition converges. For $ k<\infty $, there exists $ M>0 $ such that $ d_i(\gamma_1,\gamma_2)\leqslant M $, for any $ i\in\Z $, since $ \gamma_1([i,i+1]), \gamma_2([i,i+1]) $ are closed subsets of $ \Lambda_{k} $ which is compact. So $ \rho_{k}(\gamma_1,\gamma_2)\leqslant M\left( 1 + 2\sum\limits_{i=1}^{\infty}\frac{1}{2^i}\right)=3M $, i.e. it converges for any $ \gamma_1, \gamma_2 \in \overline{\Omega}_{k}$. 
	
	For the case $ k=\infty $, note that, $ \Lambda_{\infty} $ is contained in the horizontal strip $ \R\times [-2,2] $. Then, for any $ \gamma\in\overline{\Omega}_{\infty} $, the arcs $ \gamma([i-1,i]), \gamma([i,i+1])\subset [\lambda, \lambda+2]\times[-2,2]  $ with diameter $ 2\sqrt{5} $. Then, for every $ i\in\Z $:
	\[d_{H}(\gamma([i-1,i]),\gamma([i,i+1])) \leqslant 2\sqrt{5} \]
	
	Now, given $ \gamma_{1},\gamma_{2}\in\overline{\Omega}_{\infty} $ for every $ i\in\Z $:

	\begin{align*}
	d_{i}(\gamma_{1},\gamma_{2}) &= 
	d_{H}(\gamma_{1}([i,i+1]),\gamma_{2}([i,i+1]))\\
	&\leqslant 
	\begin{aligned}[t]
	d_{H}(\gamma_{1}([i,i+1]),\gamma_{1}([i-1,i]))+d_{H}(\gamma_{1}([i-1,i]),\gamma_{2}([i-1,i]))\\
	+d_{H}(\gamma_{2}([i-1,i]),\gamma_{2}([i,i+1]))
	\end{aligned}\\
	&\leqslant d_{H}(\gamma_{1}([i-1,i]),\gamma_{2}([i-1,i]))+4\sqrt{5}\\
	&\leqslant d_{i-1}(\gamma_{1},\gamma_{2})+4\sqrt{5}	
	\end{align*}
	
	In an analogous way, can be shown that 
	\[ d_{i-1}(\gamma_{1},\gamma_{2})\leqslant d_{i}(\gamma_{1},\gamma_{2})+4\sqrt{5} \]

	Recursively, follows that for every $ i\in\Z $:
	\[ d_{i}(\gamma_1,\gamma_2)\leqslant d_{0}(\gamma_1,\gamma_2)+4\sqrt{5}\left|i\right| \] 
	
	So we have
	\begin{align*}
	\rho_{\infty}(\gamma_1,\gamma_2) 
	&= \sum\limits_{i=-\infty}^{\infty}\frac{d_i(\gamma_1,\gamma_2)}{2^{|i|}} \leqslant \sum\limits_{i=-\infty}^{\infty}\frac{d_0(\gamma_1,\gamma_2)+4\sqrt{5}\left|i\right|}{2^{|i|}}\\
	&\leqslant 3 d_{0}(\gamma_1,\gamma_2)+8\sqrt{5}\sum\limits_{i=1}^{\infty}\frac{i}{2^{i}}<\infty.
	\end{align*}	
	
	So $ \rho_{k} $ is well defined for every $ k\in\N\cup\{\infty\} $. Let us show it is a metric: every summand in the definition of $ \rho_{k} $ is non-negative, then $ \rho_{k}(\gamma_1,\gamma_2)=0 $ if and only if $ d_i(\gamma_1,\gamma_2)=0$ for every $ i\in\Z $. Hence $ \gamma_1([i,i+1])=\gamma_2([i,i+1]),$ for all $ i\in\Z$. Then $ \gamma_1=\gamma_2 $.
	
	From $ d_i(\gamma_1,\gamma_2)=d_i(\gamma_2,\gamma_1) $, we obtain $ \rho_{k}(\gamma_1,\gamma_2)=\rho_{k}(\gamma_2,\gamma_1) $. 
	
	Now, for every $ i\in\Z $, we get: 
	\[ d_i(\gamma_1,\gamma_2)\leqslant d_i(\gamma_1,\gamma_3)+d_i(\gamma_3,\gamma_2) \Rightarrow\]
	\[ \sum\limits_{i=-N}^{N}\frac{d_i(\gamma_1,\gamma_2)}{2^{|i|}} \leqslant \sum\limits_{i=-N}^{N}\frac{d_i(\gamma_1,\gamma_3)}{2^{|i|}} + \sum\limits_{i=-N}^{N}\frac{d_i(\gamma_3,\gamma_2)}{2^{|i|}}, \forall N \Rightarrow \]
	
	\[ \rho_{k}(\gamma_1,\gamma_2)\leqslant \rho_{k}(\gamma_1,\gamma_3)+\rho_{k}(\gamma_3,\gamma_2) \]
	
	Therefore, $ (\overline{\Omega}_{k},\rho_{k}) $ is a metric space.
\end{proof}

Let $ \overline{T_{1}}:\overline{\Omega}_k \to\overline{\Omega}_k $ the function induced by $ T_{1} $, that is, $ \overline{T_{1}}(\overline{\gamma})=\overline{T_{1}(\gamma)} $. It does not depend on the representative, for if $ s(\gamma_{1})=s(\gamma_{2})=(s_{j})_{j\in\Z} $, then, for all $ j\in\Z $:
\[\gamma_{1}(j),\gamma_{2}(j)\in I_{s_{j}} \Rightarrow \gamma_{1}(j+1),\gamma_{2}(j+1)\in I_{s_{j+1}} \]
\[\Rightarrow T_{1}(\gamma_{1})(j),T_{1}(\gamma_{2})(j)\in I_{s_{j+1}} \Rightarrow s(T_{1}(\gamma_{1}))=s(T_{1}(\gamma_{2})). \]

\begin{proposition}
	The function $\overline{T_1} $ given above is a homeomorphism.
\end{proposition}

\begin{proof} 
	The function $ T_1 $ clearly is invertible, with inverse $ \left(T_1\right)^{-1}(\gamma)(.)=\gamma(.-1) $, and it is straightforward to see that $ \overline{T_{1}}^{-1}=\overline{T_{1}^{-1}} $.
	
	Now,  
	\begin{align*}
	d_i(\overline{T_1}(\gamma_1),\overline{T_1}(\gamma_2))&= d_H(T_1(\gamma_1)([i,i+1]), T_1(\gamma_2)([i,i+1])) =\\
	&= d_H(\gamma_1([i+1,i+2]), \gamma_2([i+1,i+2]))=d_{i+1}(\gamma_1,\gamma_2).
	\end{align*}
	
	Then:
	\begin{align*}
	\rho(\overline{T_1}(\gamma_1),\overline{T_1}(\gamma_2))&= \sum_{i=-\infty}^{\infty}\frac{d_i(T_1(\gamma_1),T_1(\gamma_2))}{2^{|i|}}
	=\sum_{i=-\infty}^{\infty}\frac{d_{i+1}(\gamma_1,\gamma_2)}{2^{|i|}}=\\
	&= \lim\limits_{k\to\infty}\left(\sum_{i=0}^{k}\frac{d_{i+1}(\gamma_1,\gamma_2)}{2^{i}}+ \sum_{i=-k}^{-1}\frac{d_{i+1}(\gamma_1,\gamma_2)}{2^{-i}} \right)=\\
	&=\lim\limits_{k\to\infty}\left( 2\sum_{i=0}^{k}\frac{d_{i+1}(\gamma_1,\gamma_2)}{2^{i+1}}+ \frac{1}{2}\sum_{i=-k}^{-1}\frac{d_{i+1}(\gamma_1,\gamma_2)}{2^{-i-1}} \right)\leqslant\\
	&\leqslant 2\lim\limits_{k\to\infty}\left( \sum_{j=1}^{k+1}\frac{d_{j}(\gamma_1,\gamma_2)}{2^{j}}+ \sum_{j=-k+1}^{0}\frac{d_{j}(\gamma_1,\gamma_2)}{2^{-j}} \right)=\\
	&\leqslant 2 \sum_{j=-\infty}^{\infty}\frac{d_j(\gamma_1,\gamma_2)}{2^{|j|}}=2\rho(\gamma_1,\gamma_2).
	\end{align*}
	
	Hence $ T_1 $ is continuous. The proof of continuity of the inverse is analogous.
\end{proof}

Now let $ \overline{s}:\overline{\Omega}_{k}\to \{0,1,\dots,2k-3\}^{\Z}  $ the function induced by $ s $, that is, $ \overline{s}(\gamma)=s(\gamma) $. It does not depend on the representative chosen, because of the equivalence relation and it is one-to-one.

\begin{proposition}\label{s-homeo}
	The map $ \overline{s} $ is a homeomorphism onto its image.
\end{proposition}

\begin{proof}
	\begin{description}
		\item[Case $ k<\infty $]
		Put $ \mu = \min\{d_{H}(I_{l},I_{j})\,\,|\,\, l\neq j, $ and $ l,j=1,\dots,2k-3\} $ and $ \gamma_{1},\gamma_{2}\in \overline{\Omega}_{k} $. Suppose $ \rho(\gamma_{1},\gamma_{2})<\frac{\mu}{2^N} $. Then, for any $ -N\leqslant i\leqslant N $:
		\[ d_{i}(\gamma_{1},\gamma_{2})<\mu \] 
		because, on the contrary, we have $ \rho(\gamma_{1},\gamma_{2})=\sum\frac{d_{i}(\gamma_{1},\gamma_{2})}{2^{|i|}} \geqslant \frac{\mu}{2^{N}} $. 
		
		Now 
		\[d_{i}(\gamma_{1},\gamma_{2})<\mu \Rightarrow d_{i}(\gamma_{1},\gamma_{2})=0\] 
		and therefore, $ \gamma_{1}((i,i+1))=\gamma_{2}((i,i+1)) $ implying that $ s_{i}(\gamma_{1})=s_{i}(\gamma_{2}) $, for all $ -N\leqslant i \leqslant N $. So
		\begin{align*}
			d(s(\gamma_{1}),s(\gamma_{2}))&=\sum\limits_{i\in\Z}\frac{|s_{i}(\gamma_{1})-s_{i}(\gamma_{2})|}{2^{|i|}}\\ &=\sum\limits_{\substack{i\in\Z\\|i|>N}}\frac{|s_{i}(\gamma_{1})-s_{i}(\gamma_{2})|}{2^{|i|}}\leqslant\frac{2k-3}{2^{N-1}}.
		\end{align*}
		
		This proves that $ \overline{s} $ is continuous.
		
		The same argument reverses itself in order to show $ \overline{s} $ is open: let $ \gamma_{1},\gamma_{2}\in\overline{\Omega}_{k} $, and $ N\in\N $
		\[ d(\overline{s}(\gamma_{1}),\overline{s}(\gamma_{2}))<\frac{1}{2^{N}} \Rightarrow s_{i}(\gamma_{1})=s_{i}(\gamma_{2}), \forall -N\leqslant i \leqslant N \]
		
		Then $ \gamma_{1}([i,i+1])=\gamma_{2}([i,i+1]) $, for all $ -N\leqslant i\leqslant N  $. Hence
		\[ \rho(\gamma_{1},\gamma_{2})= \sum_{\substack{i\in\Z\\|i|>N}}\frac{d_{i}(\gamma_{1},\gamma_{2})}{2^{|i|}} \leqslant\frac{M}{2^{N-1}}. \]  
		where $ M>0 $, such that diam$(\Lambda_{k})<M $.

		\item[Case $ k=\infty $] Again, let $ \mu = \min\{d_{H}(I_{l},I_{j})\,\,|\,\, l\neq j, $ and $ l,j\in\Z\} $ and $ \gamma_{1},\gamma_{2}\in \overline{\Omega}_{\infty} $, such that $ \rho(\gamma_{1},\gamma_{2})<\frac{\mu}{2^N} $. Then, for any $ -N\leqslant i\leqslant N $ we get:
		\[ d_{i}(\gamma_{1},\gamma_{2})<\mu \Rightarrow d_{i}(\gamma_{1},\gamma_{2})=0, \] 
		which implies $ s_{i}(\gamma_{1})=s_{i}(\gamma_{2}) $, for all $ -N\leqslant i \leqslant N $. So
		\begin{align*}
			d(s(\gamma_{1}),s(\gamma_{2}))
			&=\sum\limits_{\substack{i\in\Z\\|i|>N}}\frac{|s_{i}(\gamma_{1})-s_{i}(\gamma_{2})|}{2^{|i|}}\\
			&\leqslant \sum\limits_{\substack{i\in\Z\\|i|>N}}\frac{|s_{0}(\gamma_{1})-s_{0}(\gamma_{2})|+4|i|}{2^{|i|}}=\\
			&\leqslant 2\sum\limits_{i=N+1}^{\infty}\frac{4i}{2^{i}}=\frac{1}{2^{N-4}}.
		\end{align*}
		
		If $ \gamma_{1},\gamma_{2}\in\overline{\Omega}_{\infty} $ and $ N\in\N $ are such that $ d(\overline{s}(\gamma_{1}),\overline{s}(\gamma_{2}))<\frac{1}{2^{N}} $, then for all $ -N\leqslant i \leqslant N $:
		\[ s_{i}(\gamma_{1})=s_{i}(\gamma_{2}) \Rightarrow \gamma_{1}([i,i+1])=\gamma_{2}([i,i+1]) \]
	
		Hence
		\begin{align*}
			\rho(\gamma_{1},\gamma_{2})&= \sum_{\substack{i\in\Z\\|i|>N}}\frac{d_{i}(\gamma_{1},\gamma_{2})}{2^{|i|}} \leqslant\sum_{\substack{i\in\Z\\|i|>N}}\frac{d_{0}(\gamma_{1},\gamma_{2}) +4\sqrt{5}|i|}{2^{|i|}}\\
			&\leqslant \frac{8\sqrt{5}}{2^{N}}\sum\limits_{i=1}^{\infty}\frac{i}{2^{i}}=\frac{\sqrt{5}}{2^{N-4}}.
		\end{align*}
	\end{description}

	Therefore $ \overline{s} $ in a homeomorphism over its image.
\end{proof}


Now we are in conditions to prove items (i) to (iii) of Theorem \ref{teo 1}.

\subsection{Proof of item \emph{(i)} of Theorem \ref{teo 1}}
\begin{proposition}
	The function $ \overline{s}:\overline{\Omega}_{2}\to\{0,1\}^{\Z} $ is a conjugation between $ \overline{T_{1}} $ and $ \sigma $, i.e., $ \overline{s} \circ \overline{T_1}=\sigma\circ \overline{s} $  (see Figure \ref{diagram1}).
\begin{figure}[H]
	\begin{center}
		\begin{tikzpicture}
		\node (A) {$\overline{\Omega}_{2}$};
		\node (B) [right of=A] {$\overline{\Omega}_{2}$};
		\node (C) [below of=A] {$\{0,1\}^{\Z}$};
		\node (D) [below of=B] {$\{0,1\}^{\Z}$};
		\large\draw[->] (A) to node {\mbox{{\footnotesize $\overline{T_1}$}}} (B);
		\large\draw[->] (C) to node {\mbox{{\footnotesize $\sigma$}}} (D);
		\large\draw[->] (A) to node {\mbox{{\footnotesize $\overline{s}$}}} (C);
		\large\draw[->] (B) to node {\mbox{{\footnotesize $\overline{s}$}}} (D);
		\end{tikzpicture}
	\end{center}
	\caption{}
	\label{diagram1}
\end{figure}
\end{proposition}

\begin{proof}
	Let us show that $ \overline{s}(\overline{\Omega}_{2})=\{0,1\}^{\Z} $. Given $ (s_{j})\in\{0,1\}^{\Z} $, construct $ \gamma $ by concatenating the arcs $ I_{0} $ and $ I_{1} $ according to $ (s_{j}) $, so $ \gamma $ is a trajectory with $ \gamma(0)=p_{1} $, and $ \gamma((0,1))=I_{s_{0}} $. Then $ \gamma(1)=p_{1} $ and $ \gamma((1,2))= I_{s_{1}} $. In general, for all $ j\in\Z $, $ \gamma(j)=p_{1} $ and $ \gamma((j,j+1))=I_{s_{j}} $. By Definition \ref{definicao trajetoria global}, $ \gamma $ is a global trajectory of $ Z_{2} $ and therefore $ \gamma\in\Omega_{2} $. Moreover, $ s(\gamma)=\overline{s}(\gamma)=(s_{j})_{j\in\Z} $.
	
	Now, for the commutative part, let $ \gamma\in\overline{\Omega}_{2} $, $ (a_j)_{j\in\Z}=\overline{s}(\gamma) $, and $ (b_j)_{j\in\Z}=\overline{s}(\overline{T_1}(\gamma)) $, then:
	\begin{equation*}
	b_j=\left\{
	\begin{array}{l} 
	0,\quad $if$ \quad \overline{T_1}(\gamma)(j) \in I_0 \\ 
	1,\quad $if$ \quad \overline{T_1}(\gamma)(j)  \in I_1
	\end{array}\right.=
	\left\{
	\begin{array}{l} 
	0,\quad $if$ \quad \gamma(j+1) \in I_0 \\ 
	1,\quad $if$ \quad \gamma(j+1)  \in I_1
	\end{array}\right.=a_{j+1}
	\Rightarrow
	\end{equation*}
	$ \Rightarrow (b_j)_{j\in\Z}=\sigma((a_j)_{j\in\Z}), $
	i.e., $ \overline{s} \circ \overline{T_1}=\sigma\circ \overline{s} $.	 
\end{proof}

\subsection{Proof of item \emph{(ii)} of Theorem \ref{teo 1}}

%
%
%
%

\begin{proposition}
	The function $ \overline{s}:\overline{\Omega}_{k}\to \overline{s}(\overline{\Omega}_{k}) $ is a conjugation between $ \overline{T_{1}} $ and $ \sigma $, i.e., $ \overline{s} \circ \overline{T_1}=\sigma\circ \overline{s} $  (see Figure \ref{diagram2}).
\begin{figure}[H]
	\begin{center}
		\begin{tikzpicture}
		\node (A) {$\overline{\Omega}_{k}$};
		\node (B) [right of=A] {$\overline{\Omega}_{k}$};
		\node (C) [below of=A] {$\overline{s}(\overline{\Omega}_{k})$};
		\node (D) [below of=B] {$\overline{s}(\overline{\Omega}_{k})$};
		\large\draw[->] (A) to node {\mbox{{\footnotesize $\overline{T_1}$}}} (B);
		\large\draw[->] (C) to node {\mbox{{\footnotesize $\sigma$}}} (D);
		\large\draw[->] (A) to node {\mbox{{\footnotesize $\overline{s}$}}} (C);
		\large\draw[->] (B) to node {\mbox{{\footnotesize $\overline{s}$}}} (D);
		\end{tikzpicture}
	\end{center}
	\caption{}
	\label{diagram2}
\end{figure}
\end{proposition}

\begin{proof}
	Let $ \gamma\in\overline{\Omega}_{k} $, $ (a_j)_{j\in\Z}=\overline{s}(\gamma) $, and $ (b_j)_{j\in\Z}=\overline{s}(\overline{T_1}(\gamma)) $. Then:
	\begin{align*}
	b_j & = m,\quad \text{if} \quad \overline{T_1}(\gamma)(j) \in I_m \\
	& = m,\quad \text{if} \quad \gamma(j+1) \in I_m\\ 
	& = a_{j+1}
	\Rightarrow
	\end{align*}
	$ \Rightarrow (b_j)_{j\in\Z}=\sigma((a_j)_{j\in\Z}) $.
	
	It remains to show that $ \overline{s}(\overline{\Omega}_{k}) $ is a subshift. First, $ \overline{s} $ is continuous and $ \overline{\Omega}_{k} $ is compact. Then $ \overline{s}(\overline{\Omega}_{k}) $ is closed in $ \mathcal{A}_{k}^{\Z} $. And, By Proposition \ref{s-homeo}, there exists $ \overline{s}^{-1}:\overline{s}(\overline{\Omega}_{k})\to\overline{\Omega}_{k} $ and it is continuous. By the first part of this proof, we have $ \sigma= \overline{s}\circ\overline{T_{1}}\circ\overline{s}^{-1} $, which proves the invariant part. 
	
	Hence $ \overline{s} $ is a conjugation between both systems.
\end{proof}

\subsection{Proof of item \textit{(iii)} of Theorem \ref{teo 1}}

\begin{proposition}
	The function $ \overline{s}:\overline{\Omega}_{\infty}\to \overline{s}(\overline{\Omega}_{\infty})\subset \Theta_{\infty} $ is a conjugation between $ \overline{T_{1}} $ and $ \sigma $, i.e., $ \overline{s} \circ \overline{T_1}=\sigma\circ \overline{s} $  (see Figure \ref{diagram3}).
	\begin{figure}[H]
		\begin{center}
			\begin{tikzpicture}
			\node (A) {$\overline{\Omega}_{\infty}$};
			\node (B) [right of=A] {$\overline{\Omega}_{\infty}$};
			\node (C) [below of=A] {$\Theta_{\infty}$};
			\node (D) [below of=B] {$\Theta_{\infty}$};
			\large\draw[->] (A) to node {\mbox{{\footnotesize $\overline{T_1}$}}} (B);
			\large\draw[->] (C) to node {\mbox{{\footnotesize $\sigma$}}} (D);
			\large\draw[->] (A) to node {\mbox{{\footnotesize $\overline{s}$}}} (C);
			\large\draw[->] (B) to node {\mbox{{\footnotesize $\overline{s}$}}} (D);
			\end{tikzpicture}
		\end{center}
		\caption{}
		\label{diagram3}
	\end{figure}
\end{proposition}

\begin{proof}
%

Let $ \gamma\in\overline{\Omega}_{\infty} $, $ (a_j)_{j\in\Z}=\overline{s}(\gamma) $, and $ (b_j)_{j\in\Z}=\overline{s}(\overline{T_1}(\gamma)) $. Then:
\begin{align*}
b_j & = m,\quad \text{if} \quad \overline{T_1}(\gamma)(j) \in I_m \\
& = m,\quad \text{if} \quad \gamma(j+1) \in I_m\\ 
& = a_{j+1}
\Rightarrow
\end{align*}
$ \Rightarrow (b_j)_{j\in\Z}=\sigma((a_j)_{j\in\Z}) $.

	By Proposition \ref{s-homeo}, there exists $ \overline{s}^{-1}:\overline{s}(\overline{\Omega}_{\infty})\to\overline{\Omega}_{\infty} $ and it is continuous and then $ \sigma= \overline{s}\circ\overline{T_{1}}\circ\overline{s}^{-1} $, which proves that $ \sigma(\overline{s}(\overline{\Omega}_{\infty}))=\overline{s}(\overline{\Omega}_{\infty}) $.
\end{proof}

\subsection{Proof of item \emph{(iv)} of Theorem \ref{teo 1}}

Consider the PSVF
\begin{equation}\label{Eq Campo feijao}
Z(x,y)=\left\{
\begin{array}{l} 
X(x,y)= (1, -2x) ,\quad $for$ \quad y \geq 0 \\ 
Y(x,y)= (-2,-4x^3+2x),\quad $for$ \quad y \leq 0.
\end{array}
\right.
\end{equation} 

This PSVF has a compact invariant set $\Lambda=\{(x,y)\in\R^2:-1\leqslant x\leqslant 1$ and $x^4/2+x^2/2\leqslant y\leqslant 1-x^2 \}$. Furthermore, $\Lambda$ is a chaotic non-trivial minimal set for $Z$ (see \cite{BCEminimal}). As before, we will consider $ \Omega=\{\gamma $ global trajectory of $ Z\vert_{\Lambda}\} $ the set of all possible trajectories contained in the set $ \Lambda $. 

In cases \emph{(i)-(iii)} of Theorem \ref{teo 1}, one could consider the set $ \{p_{j}\} $, as a "recurrent" set, in the sense that every trajectory visits it and returns to it an infinity amount of times. Moreover, the time between visits is exactly 1 (see Proposition \ref{flight-time}), and that was the basic setting that allowed the previous constructions. In this case we consider the set $ K=\{0\}\times(0,1] $, since every local trajectory intersects it transversally, and every point of $ K $ reaches $ p=(0,0) $ in finite time (positive and negative), thus every trajectory visits it an infinity amount of times. Thus, the set $ K $ will play the role of $ \{p_{j}\} $ in previous cases. This time, we cannot adjust the expression of $ Z $ in order to make the time between visits equal some constant. So we define the function: $ \eta:\Omega\to \R $, as $ \eta(\gamma)=\min\{t> 0\,\,|\,\, \gamma(t)\in K\} $ to be the least positive value of $ t $ for which $ \gamma(t)\in K $. For simplicity we may denote $ \eta_{\gamma}=\eta(\gamma) $. 


Now define the map $\mathcal{T}:\Omega\to\Omega $, as $ \mathcal{T}(\gamma)(.)=\gamma(.+\eta_{\gamma}) $. This map takes a trajectory to another one that starts one loop ahead. 

%


\begin{lemma}\label{tempos de batida}
	Given a global trajectory $\gamma \in \Omega$, there exists a bi-infinite increasing sequence $ (t_{j}^{\gamma})_{j\in\Z} $, such that $\gamma(t_{j}^{\gamma})\in K$.
\end{lemma}
\begin{proof}
	We define $ t_{0}= \max\{t\leqslant 0\,\,|\,\, \gamma(t)\in K\} $, and for $ j>0 $, let $ t_{j}^{\gamma} = \sum\limits_{k=0}^{j-1}\eta_{\mathcal{T}^{k}(\gamma)} $. For $ j<0 $, consider $ \widetilde{\eta}(\gamma)= \max\{t < 0\,\,|\,\, \gamma(t)\in K\} $, and the analogous for $ \mathcal{T} $, that is, $ \widetilde{\mathcal{T}}:\Omega\to\Omega $, $ \widetilde{\mathcal{T}}(\gamma)(.)=\gamma(.+\widetilde{\eta}_{\gamma}) $. Then and $ t_{j}=t_{0} + \sum\limits_{k=0}^{-j-1}\widetilde{\eta}_{\widetilde{\mathcal{T}}^{k}(\gamma)}$.
	
	Given the sequence constructed above, $ \gamma(t_{0}^{\gamma})\in K $, 
	\[ \gamma(t_{1}^{\gamma})=\gamma(\eta_{\gamma})=\mathcal{T}(\gamma)(0)\in K \] 
	When $ j>1 $, we get
	
	\begin{equation*}
	\gamma(t_{j}^{\gamma})
	\begin{aligned}[t]
	&=\gamma(\eta_{\gamma}+\eta_{\mathcal{T}(\gamma)}+\hdots+\eta_{\mathcal{T}^{j-1}(\gamma)}) = \mathcal{T}(\gamma)(\eta_{\mathcal{T}(\gamma)}+\hdots+\eta_{\mathcal{T}^{j-1}(\gamma)})= \\ &=\mathcal{T}^{2}(\gamma)(\eta_{\mathcal{T}^{2}(\gamma)}+\hdots+\eta_{\mathcal{T}^{j-1}(\gamma)})=\hdots= \mathcal{T}^{j-1}(\gamma)(\eta_{\mathcal{T}^{j-1}(\gamma)})=\\
	&=\mathcal{T}^{j}(\gamma)(0) \in K.
	\end{aligned}
	\end{equation*}
	
	For $ j<0 $ it is analogous.
\end{proof}

\begin{lemma}\label{tempos de batida do retorno}
	Given a global trajectory $\gamma \in \Omega$ and $ (t_{j}^{\gamma}) $  given above, consider the sequence $ (t_{j}^{\mathcal{T}(\gamma)}) $. It holds $ t_{j}^{\mathcal{T}(\gamma)}=t_{j+1}^{\gamma}-\eta_{\gamma} $.
\end{lemma}
\begin{proof}
	In fact,

	\[ t_{0}^{\mathcal{T}(\gamma)}=0=\eta_{\gamma}-\eta_{\gamma}=t_{1}^{\gamma}-\eta_{\gamma} \]
	
	For $ j>0 $:
	\[ t_{j}^{\mathcal{T}(\gamma)}=\sum\limits_{k=0}^{j-1}\eta_{\mathcal{T}^{k+1}(\gamma)}= \sum\limits_{k=0}^{j}\eta_{\mathcal{T}^{k}(\gamma)}-\eta_{\gamma}=t_{j+1}^{\gamma}-\eta_{\gamma} \]
	
	And if $ j<0 $:
	\[ t_{j}^{\mathcal{T}(\gamma)}=t_{0} + \sum\limits_{k=0}^{-j-1}\widetilde{\eta}_{\widetilde{\mathcal{T}}^{k}(\mathcal{T}(\gamma))} = \sum\limits_{k=0}^{j}\eta_{\mathcal{T}^{k}(\gamma)}-\eta_{\gamma}=t_{j+1}^{\gamma}-\eta_{\gamma} \]
\end{proof}


\begin{definition} 
	Consider $ y:K\to (0,1] $ the projection on $y$-coordinate (that is $ y(0,\beta)=\beta $), and define the \emph{itinerary map} $ s:\Omega\to (0,1]^{\Z} $, as $ s(\gamma)=(s_{j}(\gamma))_{j\in\Z}$, where $ s_{j}(\gamma)=y(\gamma(t_{j}^{\gamma})) $. In this way, every single trajectory of $ \Omega $ can be encoded by the $y$-coordinates of its beats on $ K $.
\end{definition}

Clearly $s$ is well-defined, by construction and it is onto, because given a sequence $ (x_{j})_{j\in\Z}\in (0,1]^{\Z} $ it is possible to construct a trajectory $ \gamma\in\Omega $ by concatenating the correct arcs in order to get $ s(\gamma)=(x_{j})_{j\in\Z} $. 

As in previous cases, there are infinitely many trajectories that describes the same curve, simply by changing the initial condition (a shift in time), and the function $ s $ takes all of these to the same sequence. So, we consider the set $ \overline{\Omega}=\faktor{\Omega}{s} $, that is, two trajectories $ \gamma_{1},\gamma_{2} $ are equivalent if $ s(\gamma_{1})=s(\gamma_{2}) $.

\begin{remark}\label{representative gamma}
	Every trajectory $ \gamma\in\Omega $ is equivalent to another one $ \gamma^{*} $ such that $ \gamma^{*}(0)\in K $. 
\end{remark}

\begin{remark}
	Another way of avoiding the infinitely many trajectories that looks like the same would be to restrict our function to a subset $ \Omega_{*}\subset\Omega $, where $ \Omega_{*}=\{\gamma $ global trajectory of $ Z\vert_{\Lambda} \,\,|\,\, \gamma(0)\in K \} $, but we chose to work in a more general setting, without restricting the dynamics as it was done in cases \emph{(i)-(iii)} of Theorem \ref{teo 1}.
\end{remark}

\begin{definition}
	Let $\overline{\gamma_{1}},\overline{\gamma_{2}}\in \overline{\Omega} $, define $ \rho: \overline{\Omega}\times\overline{\Omega}\to\R $, by
	\[ \rho(\overline{\gamma_{1}},\overline{\gamma_{2}})=\sum_{i\in\Z}\frac{d_{i}(\gamma_{1}^{*},\gamma_{2}^{*})}{2^{|i|}} \]
	where $ d_{i}(\gamma_{1}^{*},\gamma_{2}^{*})= d_{H}(\gamma_{1}^{*}([t_{i}^{\gamma_{1}^{*}},t_{i+1}^{\gamma_{1}^{*}}]), \gamma_{2}^{*}([t_{i}^{\gamma_{2}^{*}},t_{i+1}^{\gamma_{2}^{*}}])) $. That is, at each step $ i $, we take the Hausdorff distance between the $i$-th loops of both trajectories.
\end{definition}

\begin{proposition}
	$ (\overline{\Omega},\rho) $ is a metric space.
\end{proposition}
\begin{proof}
	The function $ \rho $ defined above is well-defined, since $ \Lambda $ is a bounded set, which implies there exists $ C>0 $, such that $ d_{i}(\gamma_{1}^{*},\gamma_{2}^{*})<C $, for all $ i\in\Z $.
	
	The proof of those properties for it to be a metric is analogous to the one given in Proposition \ref{prop-omega-metric}.
\end{proof}

Let $ \overline{\mathcal{T}}:\overline{\Omega}\to\overline{\Omega} $ the function induced by $ \mathcal{T} $, that is, $ \overline{\mathcal{T}}(\overline{\gamma})=\overline{\mathcal{T}(\gamma)} $. It does not depend on the representative, for if $ s(\gamma_{1})=s(\gamma_{2})=(s_{j})_{j\in\Z} $, then, for all $ j\in\Z, y(\gamma_{1}(t_{j}^{\gamma_{1}}))=y(\gamma_{2}(t_{j}^{\gamma_{2}}))=s_{j} $.

\[y(\mathcal{T}(\gamma_{1})(t_{j}^{\mathcal{T}(\gamma_{1})}))=y(\gamma_{1}(t_{j}^{\mathcal{T}(\gamma_{1})}+\eta_{\gamma_{1}}))=y(\gamma_{1}(t_{j+1}^{\gamma_{1}}-\eta_{\gamma_{1}}+\eta_{\gamma_{1}}))\]
\[ =y(\gamma_{1}(t_{j+1}^{\gamma_{1}}))=y(\gamma_{2}(t_{j+1}^{\gamma_{2}}))=y(\mathcal{T}(\gamma_{2})(t_{j}^{\mathcal{T}(\gamma_{2})})) \]
\[\Rightarrow s(\mathcal{T}(\gamma_{1}))=s(\mathcal{T}(\gamma_{2})). \]

\begin{proposition}
	The function $ \overline{\mathcal{T}}:\overline{\Omega}\to\overline{\Omega} $ is a homeomorphism.
\end{proposition}

\begin{proof} 
	The function $ \mathcal{T} $ clearly is invertible, with inverse $ \left(\mathcal{T}\right)^{-1}(\gamma)(.)=\gamma(.-\eta_{\gamma}) $, and it is straightforward to see that $ \overline{\mathcal{T}}^{-1}=\overline{\mathcal{T}^{-1}} $.
	
	Now,  
	\begin{align*}
	d_i(\mathcal{T}(\gamma_1),\mathcal{T}(\gamma_2))&=
	d_{H}(\mathcal{T}(\gamma_{1})([t_{i}^{\mathcal{T}(\gamma_{1})},t_{i+1}^{\mathcal{T}(\gamma_{1})}]), \mathcal{T}(\gamma_{2})([t_{i}^{\mathcal{T}(\gamma_{2})},t_{i+1}^{\mathcal{T}(\gamma_{2})}])) \\
	&= d_{H}(\gamma_{1}([t_{i+1}^{\gamma_{1}},t_{i+2}^{\gamma_{1}}]), \gamma_{2}([t_{i+1}^{\gamma_{2}},t_{i+2}^{\gamma_{2}}]))=d_{i+1}(\gamma_1,\gamma_2).
	\end{align*}
	
	Then:
	\begin{align*}
	\rho(\overline{\mathcal{T}}(\gamma_1),\overline{\mathcal{T}}(\gamma_2))&= \sum_{i=-\infty}^{\infty}\frac{d_i(\mathcal{T}(\gamma_1),\mathcal{T}(\gamma_2))}{2^{|i|}}
	=\sum_{i=-\infty}^{\infty}\frac{d_{i+1}(\gamma_1,\gamma_2)}{2^{|i|}}=\\
	&= \lim\limits_{k\to\infty}\left(\sum_{i=0}^{k}\frac{d_{i+1}(\gamma_1,\gamma_2)}{2^{i}}+ \sum_{i=-k}^{-1}\frac{d_{i+1}(\gamma_1,\gamma_2)}{2^{-i}} \right)=\\
	&=\lim\limits_{k\to\infty}\left( 2\sum_{i=0}^{k}\frac{d_{i+1}(\gamma_1,\gamma_2)}{2^{i+1}}+ \frac{1}{2}\sum_{i=-k}^{-1}\frac{d_{i+1}(\gamma_1,\gamma_2)}{2^{-i-1}} \right)\leqslant\\
	&\leqslant 2\lim\limits_{k\to\infty}\left( \sum_{j=1}^{k+1}\frac{d_{j}(\gamma_1,\gamma_2)}{2^{j}}+ \sum_{j=-k+1}^{0}\frac{d_{j}(\gamma_1,\gamma_2)}{2^{-j}} \right)=\\
	&\leqslant 2 \sum_{j=-\infty}^{\infty}\frac{d_j(\gamma_1,\gamma_2)}{2^{|j|}}=2\rho(\gamma_1,\gamma_2).
	\end{align*}
	
	Hence $ \mathcal{T} $ is continuous. The proof of continuity of the inverse is analogous.
\end{proof}

On the space $ (0,1]^{\Z} $ we consider the same metric from those spaces before, that is, given $ (x_{j})_{j\in\Z},(y_{j})_{j\in\Z}\in (0,1]^{\Z} $ the distance between them is the real number $ d(x,y)=\displaystyle\sum\limits_{j\in\Z}\dfrac{|x_{j}-y_{j}|}{2^{|j|}} $.

\begin{proposition}
	$ \overline{s}: \overline{\Omega} \to (0,1]^{\Z} $ is a homeomorphism.
\end{proposition}

\begin{proof}
	The function $ s $ is onto, which implies that $ \overline{s} $ is onto and clearly it is injective because it is defined on the quotient $ \faktor{\Omega}{s} $. So it remains to show the continuity part. 
	
	Note that $ |s_{j}(\gamma_{1})-s_{j}(\gamma_{2})|=||\gamma_{1}(t_{j}^{\gamma_{1}})-\gamma_{2}(t_{j}^{\gamma_{2}})||\leqslant d_{j}(\gamma_{1},\gamma_{2}) $ for all $ j\in\Z $. Then:
	
	\[ d(s(\gamma_{1}),s(\gamma_{2}))=\sum\limits_{j\in\Z}\dfrac{|s_{j}(\gamma_{1})-s_{j}(\gamma_{2})|}{2^{|j|}} \leqslant \sum_{j\in\Z}\frac{d_j(\gamma_1,\gamma_2)}{2^{|j|}}=\rho(\gamma_{1},\gamma_{2}). \]
	
	Thus $ \overline{s} $ is continuous.

		The same argument reverses itself in order to show $ \overline{s} $ is open: let $ \gamma_{1},\gamma_{2}\in\overline{\Omega}$, and $ \varepsilon>0 $ such that 
	\[ d(\overline{s}(\gamma_{1}),\overline{s}(\gamma_{2}))< \varepsilon.\] 
	
	Then $d_i (\gamma_1, \gamma_2) < M \varepsilon$, where $ M>0 $ is such that diam$(\Lambda)<M $.
	
 Hence
	\[ \rho(\gamma_{1},\gamma_{2})= \sum_{\substack{i\in\Z}}\frac{d_{i}(\gamma_{1},\gamma_{2})}{2^{|i|}} \leqslant 3 M \varepsilon. \]  
	
\end{proof}

\begin{proposition}
	Let $ \sigma: (0,1]^{\Z} \to (0,1]^{\Z} $ be the shift map. Then $ s \circ \pi=\sigma\circ s $,  i.e., the following diagram commutes (Figure \ref{diagram4}):
	
	\begin{figure}[H]
		\begin{center}
			\begin{tikzpicture}
				\node (A) {$\Omega$};
				\node (B) [right of=A] {$\Omega$};
				\node (C) [below of=A] {$(0,1]^{\Z}$};
				\node (D) [below of=B] {$(0,1]^{\Z}$};
				\large\draw[->] (A) to node {\mbox{{\footnotesize $\mathcal{T}$}}} (B);
				\large\draw[->] (C) to node {\mbox{{\footnotesize $\sigma$}}} (D);
				\large\draw[->] (A) to node {\mbox{{\footnotesize $s$}}} (C);
				\large\draw[->] (B) to node {\mbox{{\footnotesize $s$}}} (D);
			\end{tikzpicture}
		\end{center}
		\caption{a}
		\label{diagram4}
	\end{figure}

\end{proposition}
\begin{proof}
	Let $ \gamma\in\Omega $, $ (a_j)_{j\in\Z}=s(\gamma) $, and $ (b_j)_{j\in\Z}=s(\mathcal{T}(\gamma)) $. Then:
	
	\begin{equation*}
	b_j
	\begin{aligned}[t]
	&= s_{j}(\mathcal{T}(\gamma)) = y(\mathcal{T}(\gamma)(t_{j}^{\mathcal{T}(\gamma)})) = y(\gamma(t_{j}^{\mathcal{T}(\gamma)}+\eta_{\gamma}))\\ 
	&= y(\gamma(t_{j+1}^{\gamma}-\eta_{\gamma}+\eta_{\gamma})) = s_{j+1}(\gamma)=a_{j+1}
	\end{aligned}
	\end{equation*}
\end{proof}

\section{Proof of Theorem B}

\begin{definition}
	Two arcs $\overline{A \, B}$ and $\overline{C \, D}$ are parameterized by an \textbf{arc length parameterization} if there exists a  homeomorphism $h: \overline{A \, B} \rightarrow \overline{C \, D}$  such that, for all $x \in \overline{A \, B}$, then $h(x)$ is the point on $\overline{C \, D}$ satisfying\[ \frac{\mbox{length }(\overline{A \, x})}{\mbox{length }(\overline{A \, B})} =  \frac{\mbox{length }(\overline{C \, h(x)})}{\mbox{length }(\overline{C \, D})}. \] 
\end{definition}

\begin{proposition}
\begin{itemize}
	\item[(i)] Let $ \widetilde{Z_k} $ (respectively $ \widetilde{Z_{\infty}} $) be a PSVF presenting a $(k-1)$-homoclinic loop in the set $ \widetilde{\Lambda_{k}} $ (respectively $\infty$-homoclinic loop in the set $ \widetilde{\Lambda_{\infty}} $). Then $ Z_{k} $ and $ \widetilde{Z_{k}} $ restricted to $ \Lambda_{k} $ and $ \widetilde{\Lambda_{k}} $, respectively, are $\Sigma$-equivalent (respectively $ Z_{\infty} $ and $ \widetilde{Z_{\infty}} $).
	
	\item[(ii)]  The PSVF $Z$ of Theorem A, restrict to $\Lambda$, is $\Sigma$-equivalent to any PSVF $\widetilde{Z}$ presenting a compact region $\widetilde{\Lambda}$ bounded by a trajectory of $\widetilde{Z}$ passing through a invisible-visible  two-fold $\widetilde{p}$. Moreover, except for $\widetilde{p}$, the PSVF $\widetilde{Z}$ has just more two invisible tangential singularities.
\end{itemize}
	
\end{proposition}
\begin{proof}
	\begin{itemize}
		
	\item[(i)]	
		$\bullet$ Consider $Z_2$ and $\widetilde{Z_2}$ restricted to  $\Lambda_2$ and $\widetilde{\Lambda}_2$. Take $h: \Lambda_2 \to \widetilde{\Lambda}_2$ such that: $h(p_1) = \widetilde{p}_1$. Now, consider the positive arc of trajectory $\gamma^{+,p_1}_X$ of $X$ starting at $p_1$ and finishing at $r_1$. Analogously, consider the positive arc of trajectory $\widetilde{\gamma}^{+,\widetilde{p}_{1}}_{\widetilde{X}}$ of $\widetilde{X}$ starting at $\widetilde{p}_1$and finishing at $\widetilde{r}_1$.  Using the arc length parameterization  we get $h(\gamma^{+,p_1}_X)=\widetilde{\gamma}^{+,\widetilde{p}_{1}}_{\widetilde{X}}$. Now, it is enough repeat this procedure and extend $h$  to $\gamma^{-,p_1}_{X}$, $\gamma^{+,p_1}_{Y}$ and $\gamma^{-,p_1}_{Y}$.
	
		$\bullet$ For $Z_k$ and $Z_{\infty}$ it is enough repeat the previous construction changing $p_1$ by $p_j$ and $r_0$ by $p_{j+1}$ when it is necessary.
		
	\item[(ii)]	
		$\bullet$  Consider $Z$ and $\widetilde{Z}$ restricted to  $\Lambda$ and $\widetilde{\Lambda}$. Take $h: \Lambda \to \widetilde{\Lambda}$ such that: $h(p) = \widetilde{p}$. Since  $\Sigma^e=(q_1,p) \times \{ 0 \}$  and  $\widetilde{\Sigma}^e \subset \widetilde{\Sigma}$ is a continuous curve  connecting $ \widetilde{q}_1$ and $ \widetilde{p}$, by arc length parameterization we get $h(\Sigma^e)=\widetilde{\Sigma}^e$. Also, we can extend $h$ in such a way that $h(q_1)=\widetilde{q}_1$.
		
		For each $u \in \Sigma^e \cup \{q_1\}$, consider the positive arc of trajectory $\gamma^{+,u}_X$ of $X$  starting at $u$ and finishing at $v\in \Sigma^s$. Analogously, consider the positive arc of trajectory $\widetilde{\gamma}^{+,\widetilde{u}}_{\widetilde{X}}$ of $\widetilde{X}$ starting at $\widetilde{u}=h(u)$ and finishing at $\widetilde{v} \in \widetilde{\Sigma}^s$.  Using the arc length parameterization  we get $h(\gamma^{+,u}_X)=\widetilde{\gamma}^{+,\widetilde{u}}_{\widetilde{X}}$. Moreover, $h(\Sigma^s)=\widetilde{\Sigma}^s$ and $h(q_2)=\widetilde{q}_2$.
		
		For each $u \in \Sigma^e \cup \{p\}$, consider the positive arc of trajectory $\gamma^{+,u}_Y$ of $Y$  starting at $u$ and finishing at $v\in \Sigma^c_{+}$. Analogously, consider the positive arc of trajectory $\widetilde{\gamma}^{+,\widetilde{u}}_{\widetilde{Y}}$ of $\widetilde{Y}$ starting at $\widetilde{u}=h(u)$ and finishing at $\widetilde{v} \in \widetilde{\Sigma}^c_{+}$.  Using the arc length parameterization  we get $h(\gamma^{+,u}_Y)=\widetilde{\gamma}^{+,\widetilde{u}}_{\widetilde{Y}}$. Moreover, $h(\Sigma^c_{+})=\widetilde{\Sigma}^c_{+}$ and $h(s_1)=\widetilde{s}_1$.
		
		For each $u \in \Sigma^c_{+}$, it is possible to repeat the previous argument and conclude that $h(\gamma^{+,u}_X)=\widetilde{\gamma}^{+,\widetilde{u}}_{\widetilde{X}}$, $h(\Sigma^c_{-})=\widetilde{\Sigma}^c_{-}$ (and $h(s_2)=\widetilde{s}_2$).
		
		For each $u \in \Sigma^c_{-}$, consider the positive arc of trajectory $\gamma^{+,u}_Y$ of $Y$  starting at $u$ and finishing at $v\in \Sigma^s \cup \{p\}$. Analogously, consider the positive arc of trajectory $\widetilde{\gamma}^{+,\widetilde{u}}_{\widetilde{Y}}$ of $\widetilde{Y}$ starting at $\widetilde{u}=h(u)$ and finishing at $\widetilde{v} \in \widetilde{\Sigma}^s \cup \{ \widetilde{p} \}$.  Using the arc length parameterization  we get $h(\gamma^{+,u}_Y)=\widetilde{\gamma}^{+,\widetilde{u}}_{\widetilde{Y}}$. 
	\end{itemize}
\end{proof}

%

\section{Examples}\label{secao exemplos}

\begin{example}\label{Ex-3symbols}
	Consider the case $k=3$. Then $ P_{3}=-x^6+\frac{3 x^4}{2}-\frac{9 x^2}{16}+\frac{1}{16} $ with roots $ \pm 1 $ and  $ \pm \frac{1}{2} $ where the first two are simple and the latter are roots of multiplicity 2. Moreover the PSVF $ Z_{3} $ is:
	\begin{equation}
	Z_{3}(x,y)=\left\{
	\begin{array}{l} 
	X_{3}(x,y)= \left(1,-6 x^5+6 x^3-\frac{9 x}{8}\right) ,\quad $for$ \quad y \geq 0 \\ 
	Y_{3}(x,y)= \left(-1,-6 x^5+6 x^3-\frac{9 x}{8}\right),\quad $for$ \quad y \leq 0,
	\end{array}
	\right.
	\end{equation}
	and the points $ p_{1}=(-\frac{1}{2},0) $ and $ p_{2}=(\frac{1}{2},0) $ are visible-visible two fold of $ Z_{3} $. The invariant region is  the set:
	\[\Lambda_{3}=\{(x,P(x))\,\,|\,\,-1\leqslant x\leqslant 1\}\cup\{(x,-P(x))\,\,|\,\,-1\leqslant x\leqslant 1\},\]
	that is partitioned into the arcs
	\begin{align*}
	I_{0}&=\left\{(x,P(x))\,\,|\,\,-1\leqslant x< -\frac{1}{2}\right\}\cup\left\{(x,-P(x))\,\,|\,\,-1\leqslant x< -\frac{1}{2}\right\}\\
	I_{1}&=\left\{(x,P(x))\,\,|\,\,-\frac{1}{2}< x< \frac{1}{2}\right\}\\
	I_{2}&=\left\{(x,-P(x))\,\,|\,\,-\frac{1}{2}< x< \frac{1}{2}\right\}\\
	I_{3}&=\left\{(x,P(x))\,\,|\,\,\frac{1}{2}< x\leqslant 1\right\}\cup\left\{(x,-P(x))\,\,|\,\,\frac{1}{2}\leqslant x\leqslant 1\right\}
	\end{align*}
	as shows Figure \ref{three-symbols} (the arcs are blue, purple, orange and red, respectively): 
	
	\begin{figure}[h]
		\includegraphics[width=0.85\linewidth]{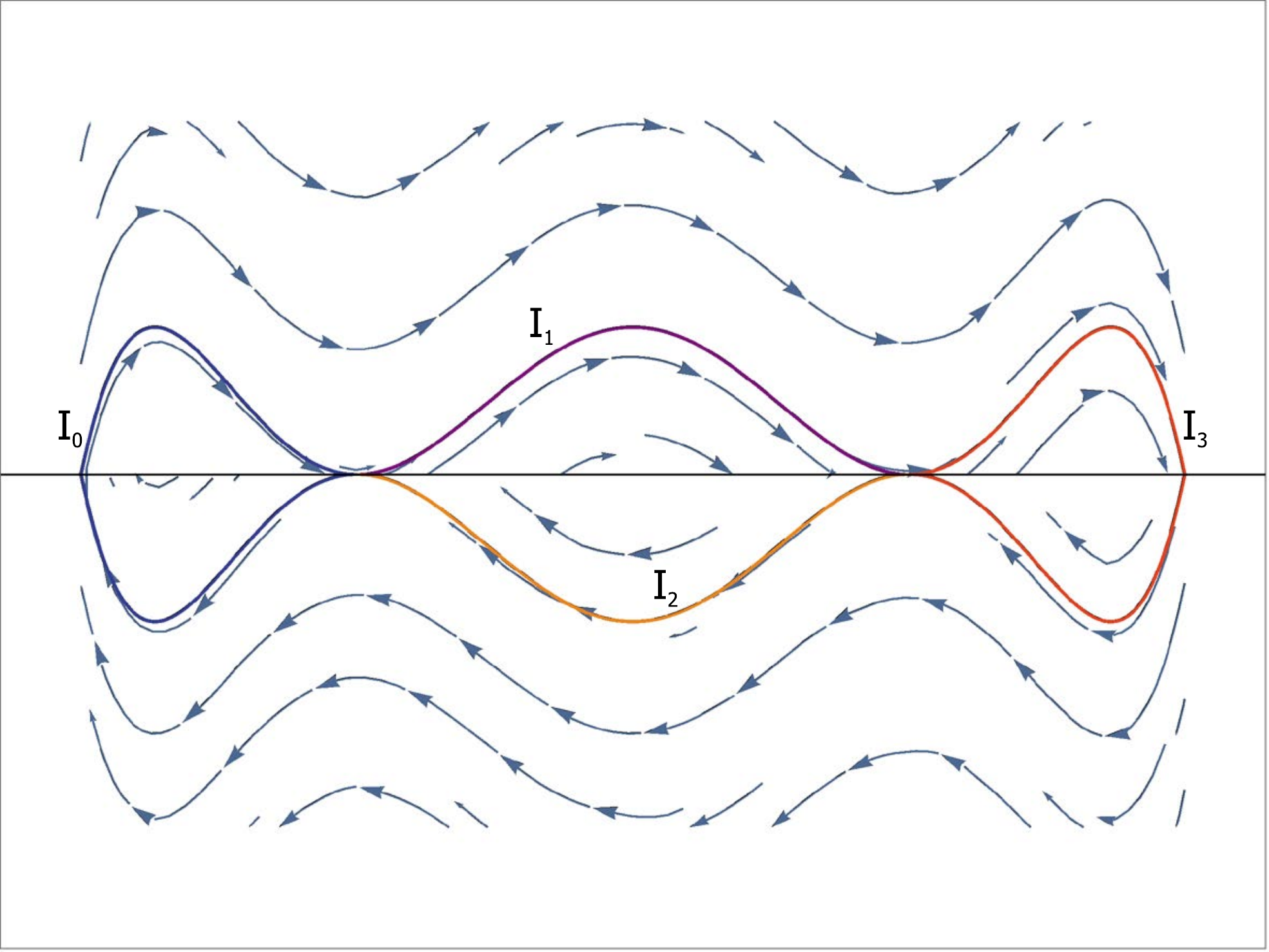}
		\caption{}
		\label{three-symbols}
	\end{figure}
	
	Now, $ \Omega_{3} $ is the set of all trajectories contained in $ \Lambda_{3} $ and $ s:\Omega_{3}\to\{0,1,2,3\}^{\Z} $. If we take $ \overline{\Omega}_{3}=\faktor{\Omega_{3}}{s} $ and the functions $ \overline{s} $ and $ \overline{T_{1}} $ as before, we have that $ \overline{s}(\overline{\Omega}_{3}) $ is a subshift of $ \{0,1,2,3\}^{\Z} $ and $ \overline{s} $ is a conjugation between $ \overline{T_{1}} $ and the shift $ \sigma $. In fact, it is easy to see that such subshift is associated to the transition matrix:
	\[M= \begin{pmatrix} 
	1 & 1 & 0 & 0 \\
	0 & 0 & 1 & 1 \\
	1 & 1 & 0 & 0 \\
	0 & 0 & 1 & 1 
	\end{pmatrix} \]
	by analysing which arc can be reached from the other ones given the orientation of the flow.
	
\end{example}

\section*{Acknowledgements}
A. A. Antunes is supported by grant 2017/18255-6, S\~{a}o Paulo Research Foundation (FAPESP). T. Carvalho is partially supported by grants  2017/00883-0 and 2019/10450-0, S\~{a}o Paulo Research Foundation (FAPESP) and by CNPq-BRAZIL grant 304809/2017-9.


\bibliographystyle{acm}
\bibliography{referencial}

\begin{thebibliography}{10}

\bibitem{BCEminimal}
{\sc Buzzi, C.~A., Carvalho, T., and Euz\'{e}bio, R.~D.}
\newblock On poincar\'{e}-bendixson theorem and non-trivial minimal sets in
  planar nonsmooth vector fields.
\newblock {\em Publicacions Matem\`{a}tiques 62}, 1 (2018), 113--131.

\bibitem{BCEchaotic}
{\sc Buzzi, C.~A., de~Carvalho, T., and Euzébio, R.~D.}
\newblock Chaotic planar piecewise smooth vector fields with non-trivial
  minimal sets.
\newblock {\em Ergodic Theory and Dynamical Systems 36}, 2 (2016), 458–469.

\bibitem{Ball-LuizTiago}
{\sc Carvalho, T., and Gon\c{c}alves, L.~F.}
\newblock Combing the hairy ball using a vector field without equilibria.
\newblock {\em Journal of Dynamical and Control Systems\/} (2019).

\bibitem{Fi}
{\sc Filippov, A.~F.}
\newblock {\em Differential Equations with Discontinuous Righthand Sides},
  first~ed., vol.~18 of {\em Mathematics and its Applications}.
\newblock Springer Netherlands, 1988.

\bibitem{goncalves_sobottka_starling_2017}
{\sc Gonçalves, D., Sobottka, M., and Starling, C.}
\newblock Two-sided shift spaces over infinite alphabets.
\newblock {\em Journal of the Australian Mathematical Society 103}, 3 (2017),
  357–386.

\bibitem{Katok03}
{\sc Hasselblat, B., and Katok, A.}
\newblock {\em A first course in dynamics: with a panorama of recent
  developments}, first~ed.
\newblock Cambridge University Press, 2003.

\bibitem{Katok95}
{\sc Katok, A., and Hasselblat, B.}
\newblock {\em Introduction to the modern theory of dynamical systems},
  first~ed.
\newblock Cambridge University Press, 1995.

\bibitem{Ott2014}
{\sc Ott, W., Tomforde, M., and Willis, P.~N.}
\newblock {\em One-sided shift spaces over infinite alphabets}, vol.~5 of {\em
  New York Journal of Mathematics. NYJM Monographs}.
\newblock State University of New York, University at Albany, Albany, NY, 2014.

\bibitem{RMCG2019}
{\sc Rodrigues, D.~S., Mancera, P. F.~A., Carvalho, T., and Gon\c{c}alves,
  L.~F.}
\newblock A mathematical model for chemoimmunotherapy of chronic lymphocytic
  leukemia.
\newblock {\em Applied Mathematics and Computation 349\/} (2019), 118--133.

\bibitem{Cancer-AMC-2019}
{\sc Rodrigues, D.~S., Mancera, P. F.~A., Carvalho, T., and Gon\c{c}alves,
  L.~F.}
\newblock Sliding mode control in a mathematical model to chemoimmunotherapy:
  the occurrence of typical singularities.
\newblock {\em Applied Mathematics and Computation in press\/} (2019).

\bibitem{Tang:2012}
{\sc {Tang}, S., {Xiao}, Y., {Wang}, N., and {Wu}, H.}
\newblock Piecewise {HIV} virus dynamic model with {CD4$^{*}$T} cell
  count-guided therapy: {I}.
\newblock {\em Journal of Theoretical Biology 308}, 7 (2012), 123 -- 134.
\newblock \url{https://doi.org/10.1016/j.jtbi.2012.05.022}.

\end{thebibliography}

\end{document}